\documentclass{article}

\usepackage{amsmath}
\usepackage{graphicx}
\usepackage{amssymb,color}
\usepackage{amsthm}
\usepackage{enumerate}
\usepackage{geometry}
\usepackage{setspace}
\usepackage{authblk}

\numberwithin{equation}{section}
\numberwithin{figure}{section}

\newtheorem{theorem}{Theorem}[section]
\newtheorem{lemma}[theorem]{Lemma}

\newtheorem{proposition}[theorem]{Proposition}
\newtheorem{remark}[theorem]{Remark}

\newtheorem*{notations}{Notations}
\newtheorem{example}[theorem]{Example}

\allowdisplaybreaks[4]

\begin{document}

\rmfamily 

\title{Sharp Asymptotic Behavior of the Steady Pressure-free Prandtl System}

\author[1]{Chen Gao\thanks{Email: gaochen@amss.ac.cn}}
\author[2]{ChuanKai Zhao\thanks{Corresponding Author. Email: chuankaizhao@amss.ac.cn}}

\affil[1]{\small School of Mathematical Sciences, University of Science and Technology of China, Hefei, 230026, P.R. China}
\affil[2]{\small School of Mathematical Sciences, Capital Normal University, Beijing, 100048, P.R. China}

\date{}

\maketitle

\begin{abstract}
	This paper investigates the asymptotic behavior of solutions to the steady pressure-free Prandtl system. By employing a modified von Mises transformation, we rigorously prove the far-field convergence of Prandtl solutions to Blasius flow. A weighted energy method is employed to establish the optimal convergence rate assuming that the initial data constitutes a perturbation of the Blasius profile. Furthermore, a sharp maximum principle technique is applied to derive the optimal convergence rate for concave initial data. The critical weights and comparison functions depend on the first eigenfunction of the linearized operator associated with the system.

\end{abstract}


\section{Introduction}

We consider the following steady Prandtl system:
\begin{equation}\label{prandtl}
	\begin{gathered}
		uu_x+vu_y-u_{yy}+p_x=0,\\
		u_x+v_y=0,
	\end{gathered}
\end{equation}
in the domain $D=\{0<x<\infty, \quad 0<y<\infty\}$, with boundary conditions
\begin{equation}\label{boundary}
	\begin{gathered}
		u(x, 0)=0, \quad u(0, y)=u_0(y), \quad v(x, 0)=0, \\
		u(x, y) \rightarrow U(x) \quad \textit{as} \quad y \rightarrow \infty.
	\end{gathered}
\end{equation}
Here $U(x)$ and $p(x)$ satisfy the Bernoulli's law governed by:
\begin{equation}
	U(x)U'(x)+p'(x)=0.
\end{equation}
In subsequent analysis, we consider the constant outer flow where without loss of generality we normalize $U(x) \equiv 1$.

The classical von Mises transformation is conventionally expressed as 
\begin{equation}\label{vM}
	X=x,\quad Y=\int_{0}^{y}u(x,y')dy'
\end{equation}
which preserves streamwise coordinates while mapping velocity profiles. Introducing a new unknown function $W(X,Y) := u^2(x,y)$, through systematic differentiation and substitution, the system (\ref{prandtl})-(\ref{boundary}) transforms into the quasilinear degenerate parabolic equation
\begin{equation}
	W_X-\sqrt{W} W_{YY}=0,
\end{equation}
defined in the transformed domain $\Omega=\{0<X<\infty, \quad 0<Y<\infty\}$, with corresponding boundary conditions
\begin{equation}
	\begin{gathered}
		W (X, 0)=0, \quad W (0, Y)=W_0(Y), \\
		W (X, Y) \rightarrow 1 \quad \textit {as} \quad Y \rightarrow \infty.
	\end{gathered}
\end{equation}

This reformulation builds upon the pioneering work of Oleinik \cite{O1}, who established the existence and uniqueness of classical solutions for the Prandtl system through the maximum principle techniques.

\begin{theorem}\label{Oleinik}
	Assume that the initial data $u_0(y)$ satisfies for some $\alpha \in (0,1)$:
	\begin{equation}
		\begin{gathered}
			u_0(y)>0 \quad \text{for}  \quad y>0,\\
			u_0(0)=0, \quad u'_0(0)>0, \quad u_0(y)\rightarrow 1 \quad \textit{as} \quad y\rightarrow \infty,\\
			u_0(y)\in C^{2,\alpha}([0,\infty]).
		\end{gathered}
	\end{equation}
Moreover, assume that for small $y$ the following compatibility condition is satisfied at the point $(0, 0)$:
\begin{equation}\label{compatibility}
	u''_0(y)=O(y^2).
\end{equation}

Then the Prandtl system \eqref{prandtl}-\eqref{boundary} admits a global classical solution $[u,v]$ in the streamwise direction with the following regularity for any fixed $X_0>0$:
\begin{enumerate}
    \item $u(x,y)$ is bounded and continuous in $[0,X_0]\times [0,\infty)$, $u>0$ for $y > 0$;
    
    \item $u_y > m > 0$ for $0 < y < y_0$, where $m$ and $y_0$ are constants;
    
    \item $u_y$, $u_{yy}$ are bounded and continuous in $[0,X_0]\times [0,\infty)$;
    
    \item $u_x$, $v$ and $v_y$ are locally bounded and continuous in $[0,X_0]\times [0,\infty)$.
\end{enumerate}
\end{theorem}

The Prandtl system \eqref{prandtl}-\eqref{boundary} admits the celebrated Blasius self-similar solutions:
\begin{equation}\label{self-solution}
	[\bar{u},\bar{v}]=\left[f'(z),\frac{1}{\sqrt{2x}}\left(zf'(z)-f(z)\right)\right],\quad z=\frac{y}{\sqrt{2(x+1)}},
\end{equation}
where the similarity function $f(z)$ satisfies the Blasius boundary value problem:
\begin{equation}\label{Blasius}
	\begin{gathered}
		f'''+ff''=0,\\
		f(0)=0, \quad f'(0)=0, \quad f'(\infty)=1.
	\end{gathered}
\end{equation}

First obtained through similarity reduction by Blasius \cite{blasius1907grenzschichten} in 1908, this ODE solution provides accurate predictions for laminar boundary layer development and serves as a fundamental benchmark in viscous flow analysis.

This paper investigates the asymptotic behavior of solutions to the Prandtl system (\ref{prandtl})-(\ref{boundary}) as $x\rightarrow \infty$. We first review key developments in the field. 

Building upon Prandtl's boundary layer theory \cite{prandtl}, Blasius \cite{blasius1907grenzschichten} obtained the first self-similar solution, while Weyl \cite{weyl1942differential} established fundamental existence theorems for the associated ordinary differential equations. A seminal breakthrough was achieved by Oleinik \cite{O1} through the application of Crocco transformation, proving local well-posedness for unsteady Prandtl equations under monotonicity assumptions. These results were systematically organized in the monograph by Oleinik and Samokhin \cite{1999Mathematical}. In a landmark study, Xin and Zhang \cite{XZ} established the global existence of weak solutions under favorable pressure gradient conditions. Notably, G{\'e}rard-Varet \cite{gerard2010ill} constructed counterexamples demonstrating ill-posedness in non-monotonic regimes. Subsequent advances by Alexandre, Wang, Xu and Yang \cite{alexandre2015well} and Masmoudi and Wong \cite{masmoudi2015local} independently re-established local well-posedness in monotonic classes using energy methods. Liu, Wang and Yang \cite{liu2017well} extended two-dimensional monotonicity results to three-dimensional flows with parallel streamlines, with subsequent work in \cite{liu2016ill} demonstrating analogous 3D ill-posedness phenomena to those in \cite{gerard2010ill}. Recently, Xin, Zhang and Zhao \cite{xin2024global} proved that the global week solution is regular by the ultra-parabolic equation theory. In analytic function spaces, Sammartino and Caflisch \cite{sammartino1998zero} first established local well-posedness, sparking subsequent developments documented in recent works by \cite{gerard2015well,paicu2021global,wang2024global,li2020well,li2022well}.

Oleinik's seminal work \cite{O1} employed the von Mises transformation to establish local well-posedness for steady Prandtl equations, extending to global existence under favorable pressure gradient conditions. Recent advances in steady inviscid limit analysis have catalyzed multifaceted research directions. Iyer \cite{iyer2020global} demonstrated asymptotic stability of Blasius solutions under perturbed initial data. Guo and Iyer \cite{guo2021regularity} established high-order regularity of local-in-$x$ Prandtl solution. Under favorable pressure conditions, Wang and Zhang \cite{wang2021global} proved that the global-in-$X$ solution is smooth. Furthermore, they showed the convergence of pressureless solutions to the Blasius profiles with concave initial profiles. Fei, Gao, Lin and Tao \cite{fei2023prandtl} obtained existence in periodic domains, while Iyer and Masmoudi \cite{iyer2022reversal} conducted stability analysis of reverse flow regimes. Gao, Zhang and Zhao \cite{GZZ} established the existence of wedge-type flows, with stability analysis by Iyer \cite{iyer2024stability}. Gao and Xin \cite{gao2023prandtl} characterized global existence and far-field asymptotics behavior for Prandtl flow in converging channels. Contemporary breakthroughs in flow separation mechanics are captured in recent works \cite{dalibard2019separation,shen2021boundary,collot2022singularity}.

Through the von Mises transformation (\ref{vM}), we define $\phi(X,Y)=u^2(x,y)-\bar{u}^2(x,y)$, which satisfies the governing equation
\begin{equation}\label{phi}
    \phi_X - u\phi_{YY} + \mathcal{A}\phi = 0,
\end{equation}
where $\mathcal{A}(X,Y)=\frac{-2\bar{u}_{yy}}{\bar{u}(\bar{u}+u)}$, with boundary conditions
\begin{equation}
    \phi(0,Y)=\phi_0(Y),\quad \phi(X,0)=\phi(X,\infty)=0.
\end{equation}

Serrin's classical result \cite{serrin1967asymptotic} established asymptotic convergence via maximum principle:
\begin{theorem}[Serrin]
    For solutions $u(x,y)$ of the pressureless Prandtl system (\ref{prandtl})-(\ref{boundary}),
    \begin{equation}
        \lim_{x\to\infty}\|u(x,y)-\bar{u}(x,y)\|_{L_y^\infty} = 0.
    \end{equation}
\end{theorem}

Iyer \cite{iyer2020global} derived explicit decay estimates for the difference $\phi(X,Y)$ and its derivatives when the initial data $\phi_0(Y)$ is small in suitable weighted Sobolev spaces with high regularities by using the energy method for the equation (\ref{phi}). 

Wang and Zhang \cite{wang2023asymptotic} obtained the decay rate
\begin{equation}
	\left|u(x,y)-\bar{u}(x,y)\right|\lesssim \frac{\ln(x+e)}{\sqrt{x+1}}e^{-\frac{cy^2}{x+1}}
\end{equation}
without smallness assumptions but with a stronger localization condition, which relies on the infinity behavior of the Blasius profile. And they established decay estimates for higher-order derivatives of $u$ under concave initial data conditions. However, the convergence rates in \cite{iyer2020global, wang2023asymptotic} are not optimal.

The recent breakthrough by Jia, Lei and Yuan \cite{jia2024sharp} established the sharp decay
\begin{equation}
    \|u(x,y)-\bar{u}(x,y)\|_{L_y^\infty} \lesssim (x+1)^{-1},
\end{equation}
when the initial data is a small localized perturbation of the Blasius profile by using the energy method. More precisely, they suppose that $u_0$ satisfies
\begin{equation}
	(1+y)(u_0-\bar{u}_0)\in L_y^1, \quad \sum_{i\leq 2}^{} \left\|(1+y) \partial_y^i(u_0-\bar{u}_0)\right\|_{L_y^2} \leq \kappa \ll 1.
\end{equation}
Furthermore, they also prove the sharp convergence rate for general initial data, with the assumption that the initial data $u_0(y)$ converge to the outer flow at the following rate:
\begin{equation}
	\left|u_0(y)-1\right|\lesssim \frac{1}{1+y^4}.
\end{equation}

We rigorously establish the sharp $\frac{1}{x}$ decay rate for solutions $u$ and obtain optimal decay estimates for arbitrary higher-order derivatives. Our framework encompasses two distinct scenarios: localized perturbations of Blasius profiles and general initial data configurations. The sharpness of these decay rates is explicitly demonstrated in Example \ref{example} through constructive verification.

\paragraph{Localized Perturbation Case:} 
\begin{theorem}\label{energy1}
Assume the initial perturbation $u_0-\bar{u}$ satisfies the weighted Sobolev bounds:
\begin{equation}\label{assumption}
    \sum_{j=0}^{2K_0}\left\| e^{\frac{y^2}{8}}\partial_y^j (u_0-\bar{u}_0) \right\|_{L^2_y} \leq \kappa(\varepsilon) \ll 1.
\end{equation}
Assume further the higher-order compatibility conditions at the corner $(x, y) = (0, 0)$. Then the following asymptotic stability results hold for all $(x,y) \in D$ with $2i+j \leq 2K_0-1$:
\begin{equation}
    \left\|\partial_x^i \partial_y^j (u-\bar{u})(x,y)\right\|_{L_y^\infty} \lesssim (x+1)^{-1-i-\frac{j}{2}}.
\end{equation}
\end{theorem}

\paragraph{General Initial Data Case:}
\begin{theorem}\label{main1}
Let $u_0(y)$ satisfy Theorem \ref{Oleinik}'s conditions with the exponential decay property:
\begin{equation}\label{suppose1}
    |u_0(y)-1| \lesssim_\varepsilon e^{-\varepsilon y^2}
\end{equation}
for $0 < \varepsilon \ll 1$. Then the following asymptotic stability result holds for all $(x,y) \in D$:
\begin{equation}
    \left\|u(x,y)-\bar{u}(x,y)\right\|_{L_y^\infty} \lesssim_\varepsilon (x+1)^{-1}.
\end{equation}
\end{theorem}

\begin{theorem}\label{main2}
Under the assumptions of Theorem \ref{main1} with additional concavity $\partial_y^2 u_0(y) \leq 0$ and higher-order regularity:
\begin{equation}\label{suppose2}
    \left|\partial_y^k u_0(y)\right| \lesssim_\varepsilon e^{-\varepsilon y^2}
\end{equation}
for $1 \leq k \leq 2K_0$ and $0 < \varepsilon \ll 1$. Assume further the higher-order compatibility conditions at the corner $(x, y) = (0, 0)$. Then the following asymptotic stability results hold for all $(x,y) \in D$ with $2i+j \leq 2K_0$:
\begin{equation}
    \left\|\partial_x^i \partial_y^j (u-\bar{u})(x,y)\right\|_{L_y^\infty} \lesssim_\varepsilon (x+1)^{-1-i-\frac{j}{2}}.
\end{equation}
\end{theorem}

\paragraph{Optimality Demonstration:}
\begin{example}\label{example}
The convergence rate in our main theorem is optimal, as demonstrated by the following explicit construction. For any $s>0$, the function 
one can check
\begin{equation}
    \bar{u}^s(x,y)=f'\left(\frac{y}{\sqrt{2(x+s)}}\right)
\end{equation}
provides an exact solution to the Prandtl equation (\ref{prandtl}), where $f$ denotes the classical Blasius Blasius similarity function. Through Lemma \ref{lemma2}'s characterization of Blasius profiles, we may select $s=s(\varepsilon)$ such that the initial data $u_s(0,y)$ satisfies both hypotheses (\ref{assumption}), (\ref{suppose1}) and \eqref{suppose2}.

Let $z_1=\frac{y}{\sqrt{2(x+s)}}$, $z_2=\frac{y}{\sqrt{2(x+1)}}$. The essential estimate follows through the analysis:
\begin{align*}
    \left\|\bar{u}^s(x,y)-\bar{u}(x,y)\right\|_{L_y^\infty}&=\left\|f'(z_1)-f'(z_2)\right\|_{L_y^\infty}=\left\|f''(\zeta) (z_1-z_2)\right\|_{L_y^\infty}\\
    &\geq \left\|z_1 f''(\zeta)\right\|_{L_y^\infty} \left(1-\frac{\sqrt{x+s}}{\sqrt{x+1}}\right) \\
    &\geq c(\varepsilon) (x+1)^{-1}
\end{align*}
for some $\zeta$ between $z_1$ and $z_2$. This establishes the sharpness of the $(x+1)^{-1}$ decay rate.

For the derivative estimates with $i,j \geq 0$:
\begin{align*}
    \left\|\partial_x^i \partial_y^j \left(\bar{u}^s-\bar{u}\right)\right\|_{L_y^\infty}&=\left\|\partial_x^i \partial_y^j \left[ f'(z_1)-f'(z_2) \right]\right\|_{L_y^\infty}\\
    &\geq (x+1)^{-i-\frac{j}{2}} \left\|z_2^{i+1} f^{(2+i+j)}(\zeta)\right\|_{L_y^\infty} \left(1-\frac{\sqrt{x+s}}{\sqrt{x+1}}\right)\\
    &\geq c(\varepsilon) (x+1)^{-1-i-\frac{j}{2}}
\end{align*}
for some $\zeta$ between $z_1$ and $z_2$, establishing the sharpness of the $(x+1)^{-1-i-\frac{j}{2}}$ decay rate.
\end{example}

Now we present the key points and main ideas of this paper. First, we introduce the modified von Mises coordinates:
\begin{equation*}
	\xi=\ln(x+1),\quad \psi=\int_{0}^{y}\frac{u(x,y')}{\sqrt{x+1}}dy'.
\end{equation*}
This coordinate transformation leads to the following equation for an unknown function $\omega(\xi,\psi)=u^2(x,y)$:
\begin{equation*}
	\omega_\xi-\sqrt{\omega}\omega_{\psi\psi}-\frac{1}{2}\psi\omega_\psi=0.
\end{equation*}
For the self-similar Blasius profile solution $\bar{u}^2(x,y)=\bar{\omega}(\psi)$, we derive
\begin{equation*}
	\sqrt{\bar{\omega}}\bar{\omega}_{\psi\psi}+\frac{1}{2}\psi\bar{\omega}_\psi=0.
\end{equation*}

Our analysis now focuses on the asymptotic behavior of $\omega$. Defining the remainder $\tilde{\omega}=\omega-\bar{\omega}$, we obtain
\begin{equation}
	\tilde{\omega}_\xi-\sqrt{\bar{\omega}}\tilde{\omega}_{\psi\psi}-\frac{1}{2}\psi\tilde{\omega}_\psi-\frac{\bar{\omega}_{\psi\psi}}{2 \sqrt{\bar{\omega}}}\tilde{\omega}=F,
\end{equation}
where $F$ is the nonlinear term (\ref{nonline}).

Thus, to establish the sharp decay estimate of $\omega$ towards $\bar{\omega}$, the most critical aspect is to determine the principal eigenvalue of the following linear operator:
\begin{equation}
	\mathcal{L}(v)=-\sqrt{\bar{\omega}} v_{\psi\psi}-\frac{1}{2}\psi v_\psi-\frac{\bar{\omega}_{\psi\psi}}{2 \sqrt{\bar{\omega}}} v.
\end{equation}
We will prove in Appendix \ref{app:eigenvalue} that $\mathcal{L}$ possesses a principal eigenvalue of $1$ with corresponding eigenfunction $\psi\bar{\omega}_\psi$.

When considering initial data constituting small localized perturbations of the Blasius profile, we employ energy methods for the weighted error equation:
\begin{equation}
	A \tilde{\omega}_\xi-A\mathcal{L}(\tilde{\omega})= A \tilde{\omega}_\xi-(\rho \tilde{\omega}_\psi)_{\psi}-\frac{\bar{\omega}_{\psi\psi}}{2 \sqrt{\bar{\omega}}} A \tilde{\omega}=AF,
\end{equation}
where $A=\frac{\rho}{\sqrt{\bar{\omega}}}$ and $\rho=exp\left(\int_{0}^{\psi}\frac{s}{2\sqrt{\bar{\omega}(s)}}ds\right)$.

The key point is the following sharp decay estimate:
\begin{equation}
	\left\|\sqrt{A} \tilde{\omega} \right\|_{L_\psi^2}^2 \lesssim e^{-\xi}\left( \left\|\sqrt{A} \tilde{\omega}_0\right\|_{L_\psi^2}^2+\left\|\sqrt{A} F e^{\xi}\right\|_{L_{\xi, \psi}^2}^2\right),
\end{equation}
supported by the critical inequality:
\begin{equation}
	\int_0^{\infty} \rho \tilde{\omega}_\psi^2 d \psi-\int_0^{\infty} \frac{\bar{\omega}_{\psi \psi}}{2 \sqrt{\bar{\omega}}} A \tilde{\omega}^2 d \psi \geq \int_0^{\infty} A \tilde{\omega}^2 d \psi,
\end{equation}
which is derived from the properties of the linear operator $\mathcal{L}$. 

This allows us to establish the stability of the solution $\tilde{\omega}$ in the weighted space $\left\|\tilde{\omega}\right\|_{\mathbb{X}}$ defined in subsequent sections. Consequently, we obtain the sharp decay estimate for $\omega$ and low-decay estimates for its derivatives. We ultimately yield sharp decay estimates through enhancement techniques.

For general initial data, we apply maximum principle techniques to the error equation:
\begin{equation}
	\tilde{\omega}_\xi-\frac{1}{2}\psi\tilde{\omega}_\psi-\sqrt{\omega}\tilde{\omega}_{\psi\psi}-\frac{\bar{\omega}_{\psi\psi}}{\sqrt{\bar{\omega}}+\sqrt{\omega}}\tilde{\omega}=0.
\end{equation}

First, we construct barrier functions reflecting the Blasius profile structure to establish preliminary decay estimates. The crucial step involves developing barrier functions capturing sharp decay rates using the principal eigenfunction $\psi\bar{\omega}_\psi$ of the linear operator $\mathcal{L}$. Additional techniques address nonlinear effects in the error equation.

\begin{remark}
	For wedge flows with $U(x)=(x+1)^m$, $m>0$, analogous analysis applies using the transformation:
	\begin{equation*}
		\xi=\ln(x+1),\quad \psi=\int_{0}^{y}\frac{u(x,y')}{(x+1)^\frac{m+1}{2}}dy'.
	\end{equation*}
	The transformed equation for $\omega(\xi,\psi)=\left(\frac{u(x,y)}{(x+1)^m}\right)^2$ becomes:
	\begin{equation*}
		\omega_\xi-\sqrt{\omega}\omega_{\psi\psi}-\frac{m+1}{2}\psi\omega_\psi +2m\omega -2m =0.
	\end{equation*}
	For Falkner-Skan profile solutions $\left(\frac{\bar{u}}{(x+1)^m}\right)^2=\bar{\omega}(\psi)$, it holds:
	\begin{equation*}
		\sqrt{\bar{\omega}}\bar{\omega}_{\psi\psi}+\frac{m+1}{2}\psi\bar{\omega}_\psi +2m-2m\bar{\omega} =0.
	\end{equation*}
	
	The remainder $\tilde{\omega}=\omega-\bar{\omega}$ satisfies:
	\begin{equation*}
		\tilde{\omega}_\xi-\sqrt{\bar{\omega}}\tilde{\omega}_{\psi\psi}-\frac{m+1}{2}\psi\tilde{\omega}_\psi-\frac{\bar{\omega}_{\psi\psi}}{2 \sqrt{\bar{\omega}}}\tilde{\omega}+2m\tilde{\omega}=F_m,
	\end{equation*}
	where $F_m$ is the nonlinear term.

	Therefore, determining the principal eigenvalue of the operator
    \begin{equation*}
        \mathcal{L}_m (v)=-\sqrt{\bar{\omega}}v_{\psi\psi}-\frac{m+1}{2}\psi v_\psi-\frac{\bar{\omega}_{\psi\psi}}{2\sqrt{\bar{\omega}}}v+2mv
    \end{equation*}
    enables extension of sharp decay estimates to arbitrary $m>0$.
\end{remark}

The remainder of the paper is structured as follows: In Section 2, we will introduce several auxiliary lemmas that will be frequently used in our proofs. In Section 3, we will establish the main Theorem \ref{energy1} via energy methods. The key point is to obtain an optimal estimate based on the properties of the linear operator $\mathcal{L}$. In Section 4, we will prove the main Theorem \ref{main1} by the maximum principle techniques. The key ingredient is to find barrier functions depending on the structure of Blasius profile. In Section 5, we will detail the proof of Theorem \ref{main2} by using the maximum principle. To ensure completeness, Appendix \ref{app:eigenvalue} contains the principal eigenvalue analysis for $\mathcal{L}$.

\begin{notations}
	In this article, we adopt the following notation conventions: We use $A\lesssim B$ means $A\leq C B$ for some absolute constant $C$. The notation $A\sim B$ means $A\lesssim B$ and $B\lesssim A$. Additionally, $A\lesssim_{\star} B$ implies $A\leq C B$ with $C$ dependent on a specific quantity $\star$, i.e. $C=C(\star)$.
\end{notations}

\section{Preliminary}

This section establishes foundational results essential for subsequent analysis. We begin by introducing the modified von Mises coordinates:
\begin{equation}\label{nvM}
	\xi=\ln(x+d), \quad \psi=\int_{0}^{y}\frac{u(x,y')}{\sqrt{x+d}}dy',
\end{equation}
where $d$ is a variable positive parameter. Defining a new unknown function $\omega(\xi,\psi)=u^2(x,y)$, this coordinate system converts the Prandtl system \eqref{prandtl}-\eqref{boundary} in $D=\{0<x<\infty, 0<y<\infty\}$ to the simplified equation:
\begin{equation}\label{omega}
	\omega_\xi-\sqrt{\omega}\omega_{\psi\psi}-\frac{1}{2}\psi\omega_\psi=0
\end{equation}
in the domain $\Omega=\{(\xi,\psi) | \ln d<\xi<\infty, 0<\psi<\infty\}$, with boundary conditions
\begin{equation}
	\begin{gathered}
		\omega(\xi, 0)=0, \quad \omega\left(\ln d, \psi\right)=\omega_0(\psi), \\
		\omega(\xi, \psi) \rightarrow 1 \quad \textit {as} \quad \psi \rightarrow \infty,
	\end{gathered}
\end{equation}
where the initial profile satisfies:
\begin{equation}
	\omega_0\left(\int_{0}^{y}\frac{u_0(y')}{d}dy'\right)=u_0^2(y).
\end{equation}

Building on the proof of Theorem 2.1.1 in \cite{1999Mathematical}, we can easily derive the following properties of $\omega(\xi,\psi)$:
\begin{lemma}\label{ome}
	The transformed solution $\omega(\xi,\psi)$ satisfies:
	\begin{equation*}
		\begin{gathered}
			\left|\omega\right|\leq M,\quad \left|\omega_\psi\right|\leq M, \quad \left|\sqrt{\omega}\omega_{\psi\psi} \right|\leq M,\\
			\omega(\xi,\psi)>0 \quad \text{for} \quad \psi>0,\\
			\omega_\psi \geq m>0 \quad \text{for} \quad 0\leq \psi \leq \psi_0,\\
			\left|\omega_\xi\right|\leq M \psi^{1-\beta} \quad \text{for} \quad 0\leq \psi \leq \psi_0,\quad 0<\beta<\frac{1}{2}.
		\end{gathered}
	\end{equation*}
\end{lemma}
\begin{proof}
	From \cite[Theorem 2.1.1]{1999Mathematical}, the original solution $W(X,Y)$ satisfies for any $X_0>0$:
	\begin{equation*}
		\begin{gathered}
			\left|W\right|\leq M, \quad \left|W_Y\right|\leq M, \quad \left|\sqrt{W}W_{YY}\right|\leq M,\\
			W(X,Y)>0 \quad \textit{for} \quad Y>0,\\
			W_Y \geq m>0 \quad \textit{for} \quad 0\leq Y \leq Y_0,\\
			\left|W_X\right|\leq M Y^{1-\beta} \quad \textit{for} \quad 0\leq Y \leq Y_0, \quad 0<\beta<\frac{1}{2}.
		\end{gathered}
	\end{equation*}

	The coordinate transformations \eqref{vM} and \eqref{nvM} induce the following differential correspondences:
	\begin{equation*}
		\begin{gathered}
			\omega(\xi,\psi)=u^2(x,y)=W(X,Y),\\
			\omega_\psi=2\sqrt{x+d} u_y=\sqrt{X+d} W_Y,\\
			\sqrt{\omega}\omega_{\psi\psi}=2(x+d)u_{yy}=(X+d)\sqrt{W}W_{YY},\\
			\omega_\xi-\frac{1}{2}\psi \omega_\psi=(X+d)W_X.
		\end{gathered}
	\end{equation*}
	These identifications directly transfer the regularity properties from $W(X,Y)$ to $\omega(\xi,\psi)$.
\end{proof}

The Blasius profile $f(z)$ with similarity variable $z = y/\sqrt{2(x+d)}$ is governed by
\begin{equation*}
	\begin{gathered}
		f'''+ff''=0,\\
		f(0)=0, \quad f'(0)=0, \quad f'(\infty)=1.
	\end{gathered}
\end{equation*}

The fundamental properties of $f(z)$ are summarized in the following lemma:
\begin{lemma}\label{f}
	The Blasius function $f(z)$ satisfies
	\begin{equation*}
		\begin{gathered}
			f(z)>0, \quad 0<f'(z)<1 \quad \text{for}\quad 0< z<\infty,\\
			f''(z)>0 \quad \text{for} \quad 0\leq z<\infty,\quad f''(0)=b_0>0,\\
			f'''(z)<0 \quad \text{for} \quad 0< z<\infty,\quad f'''(0)=0,\\
			1-f'(z)\sim N_1 z^{-1}\exp\left(-\frac{z^2}{2}-N_2 z\right), \quad f''(z)\sim z(1-f'(z)) \quad \text{as} \quad z\rightarrow \infty,
		\end{gathered}
	\end{equation*}
	where $N_1$, $N_2$ are positive constants.
\end{lemma}

The transformation $\psi = \sqrt{2}f(z)$ converts the solution to $\bar{\omega}(\psi)=\left(f'(z)\right)^2$, satisfying
\begin{equation}\label{bar}
	\sqrt{\bar{\omega}}\bar{\omega}_{\psi\psi}+\frac{1}{2}\psi\bar{\omega}_\psi=0,
\end{equation}
with boundary conditions:
\begin{equation*}
	\bar{\omega}(0)=0, \quad \bar{\omega}(\infty)=1.
\end{equation*}
The transformed function $\bar{\omega}(\psi)$ inherits the following properties:

\begin{lemma}\label{barw}
	The function $\bar{\omega}(\psi)$ satisfies
	\begin{equation*}
		\begin{gathered}
			0<\bar{\omega}(\psi)<1 \quad \text{for}\quad 0<\psi<\infty,\\
			\bar{\omega}_\psi(\psi)>0 \quad \text{for}\quad 0<\psi<\infty,\quad \bar{\omega}_\psi(0)=\sqrt{2}b_0>0,\\
			\sqrt{\bar{\omega}}\bar{\omega}_{\psi\psi} \leq 0 \quad \text{for}\quad 0\leq \psi < \infty,\\
			1-\bar{\omega}(\psi) \sim N_1\psi^{-1}exp\left(-\frac{\psi^2}{2}-N_2\psi\right) \quad \text{as} \quad \psi \rightarrow \infty,
		\end{gathered}
	\end{equation*}
	where $N_1$, $N_2$ are positive constants.
\end{lemma}

A comparison principle emerges from the structural relationship between equations (\ref{omega}) and (\ref{bar}):
\begin{lemma}\label{lemma1}
	There exist positive constants $k_1<1$ and $k_2>1$, dependent on $\omega_0$, such that 
	\begin{equation*}
		k_1\bar{\omega}(\psi)\leq \omega(\xi,\psi) \leq k_2\bar{\omega}(\psi).
	\end{equation*}
\end{lemma}

\begin{proof}
    Applying the maximum principle requires initial selection of $k_1 < 1$ and $k_2 > 1$ satisfying
	\begin{equation*}
		k_1\bar{\omega}(\psi)\leq \omega_0(\psi) \leq k_2\bar{\omega}(\psi).
	\end{equation*}

	For the lower bound, define $\varphi = \omega - k_1\bar{\omega}$. Computation reveals:
	\begin{equation*}
		\begin{aligned}
			&\partial_\xi (k_1\bar{\omega})-\sqrt{k_1\bar{\omega}}(k_1\bar{\omega})_{\psi\psi}-\frac{1}{2}\psi(k_1\bar{\omega})_\psi\\
			=&k_1(1-\sqrt{k_1})\sqrt{\bar{\omega}}\bar{\omega}_{\psi\psi} \leq 0.
		\end{aligned}
	\end{equation*}
	This difference function $\varphi$ satisfies the differential inequality:
	\begin{equation*}
		\varphi_\xi-\sqrt{\omega}\varphi_{\psi\psi}-\frac{1}{2}\psi\varphi_\psi-\frac{k_1\bar{\omega}_{\psi\psi}}{\sqrt{\omega}+\sqrt{k_1\bar{\omega}}}\varphi \geq 0,
	\end{equation*}
	with non-negative coefficient $-\frac{k_1\bar{\omega}_{\psi\psi}}{\sqrt{\omega}+\sqrt{k_1\bar{\omega}}}$. 

	Given the boundary conditions $\varphi(\xi,0) = 0$, $\varphi(\xi,\infty) = 1 - k_1 > 0$ and initial data $\varphi(\ln d, \psi) \geq 0$, the maximum principle guarantees $\varphi \geq 0$ in $\Omega$, i.e. $\omega - k_1\bar{\omega} \geq 0$ in $\Omega$. The upper bound follows similarly.
\end{proof}

Next, we investigate the fundamental properties of the operator
\begin{equation*}
    \mathcal{A}(\xi,\psi) = -\frac{\bar{\omega}_{\psi\psi}}{\sqrt{\omega} + \sqrt{\bar{\omega}}},
\end{equation*}
which plays a crucial role in our subsequent analysis.

\begin{lemma}\label{lemma2}
	For all $(\xi,\psi) \in \Omega$, the following hold:
	\begin{equation*}
		|\mathcal{A}(\xi,\psi)|\leq M,
	\end{equation*}
	where $M$ is a positive constant. Moreover, for any $\psi_0>0$, there exists a positive constant $\lambda$ such that 
	\begin{equation*}
		\mathcal{A} (\xi,\psi) \geq \lambda \quad\text{for}\quad 0 \leq \psi \leq \psi_0.
	\end{equation*}
\end{lemma}

\begin{proof}
	Lemma \ref{lemma1} provides the comparison bounds:
	\begin{equation*}
		k_1\bar{\omega}(\psi)\leq \omega(\xi,\psi) \leq k_2\bar{\omega}(\psi).
	\end{equation*}

	Noting that $f'''=\sqrt{\bar{\omega}}\bar{\omega}_{\psi\psi}\leq 0$ and $\bar{\omega}(\psi)=\left(f'(z)\right)^2$, one has the following inequalities:
    \begin{equation*}
	    -\frac{1}{1+\sqrt{k_2}}\frac{f'''}{(f')^2} \leq \mathcal{A} \leq -\frac{1}{1+\sqrt{k_1}}\frac{f'''}{(f')^2}.
    \end{equation*}

	Lemma \ref{f} yields the estimates:
	\begin{equation*}
        \frac{-f'''}{(f')^2}\leq M.
	\end{equation*}
	Near the origin $z \rightarrow 0^+$, we calculate
	\begin{equation*}
		\frac{-f'''}{(f')^2}=\frac{ff''}{(f')^2}=\frac{1}{2}.
	\end{equation*}
	Through careful selection of constants $M$, $\psi_0$, and $\lambda$, we establish the required inequalities.
\end{proof}

We now prove concavity preservation for the Prandtl system:
\begin{lemma}
	For solutions $u(x,y)$ of \eqref{prandtl}-\eqref{boundary} with concave initial data $\partial_y^2u_0 \leq 0$,
    \begin{equation*}
        \partial_y^2 u \leq 0 \quad \text{in} \ D.
    \end{equation*}
\end{lemma}
\begin{proof}
	Under the modified von Mises transformation \eqref{nvM}, the concavity condition transforms to
	\begin{equation*}
		p(\xi,\psi) := \sqrt{\omega}\omega_{\psi\psi} \leq 0.
	\end{equation*}

	Double differentiation of \eqref{omega} with respect to $\psi$ produces
	\begin{equation*}
		p_\xi-\sqrt{\omega}p_{\psi\psi}-\frac{1}{2}\psi p_{\psi}+\left(-\frac{\omega_\xi}{2\omega}+\frac{\psi\omega_\psi}{4\omega}-1\right)p=0,
	\end{equation*}
	with boundary conditions
	\begin{equation*}
		p(\xi,0)=p(\xi,\infty)=0,\quad p(\ln d,\psi)=\sqrt{\omega_0}\omega_{0\psi\psi}\leq 0.
	\end{equation*}
	
	Defining $P = e^{M\xi}p$, one has
	\begin{equation*}
		P_\xi-\sqrt{\omega}P_{\psi\psi}-\frac{1}{2}\psi P_{\psi}+\left(M -\frac{\omega_\xi}{2\omega}+\frac{\psi\omega_\psi}{4\omega}-1\right)P=0.
	\end{equation*}
	The maximum principle ensures $P \leq 0$ when $M$ is sufficiently large. The argument hinges on controlling $-\frac{\omega_\xi}{2\omega}$ and $\frac{\psi\omega_\psi}{4\omega}$. Noting that 
	\begin{equation*}
		\begin{gathered}
			\left|\frac{\psi\omega_\psi}{4\omega}\right|\lesssim 1,\\
			\left| -\frac{\omega_\xi}{2\omega}\right|\lesssim 1 \quad \textit{for}\quad \psi\geq \psi_0,
		\end{gathered}
	\end{equation*}
	where $\psi_0$ is a small positive constant. Thus the essential remaining estimate reduces to:
	\begin{equation*}
		\left| -\frac{\omega_\xi}{2\omega}\right|\lesssim 1 \quad \textit{for}\quad 0\leq \psi \leq \psi_0.
	\end{equation*}

	By Lemma \ref{ome}, one has $\left|\omega_\xi\right| \lesssim \psi^{1-\beta}$ for $0<\beta<\frac{1}{2}$. To refine the growth estimate, we define $r := \omega_\xi$ satisfying:
	\begin{equation*}
		L(r)=r_\xi-\sqrt{\omega}r_{\psi\psi}-\frac{1}{2}\psi r_\psi+\frac{\psi \omega_\psi}{4\omega}r=\frac{r^2}{2\omega},
	\end{equation*}
	with
	\begin{equation*}
		\left|\frac{r^2}{2\omega}\right|\lesssim \psi^{1-2\beta}.
	\end{equation*}
    
	Consider the barrier function $\varphi(\psi)=M\psi \left(1-\psi^\beta \right)$, where $M$ is a positive constant to be determined. One first verifies that
	\begin{equation*}
		\varphi\pm r(\xi,0)=0, \quad \varphi\pm r(\xi,\psi_0)\geq 0, \quad \varphi\pm r(0,\psi)\geq 0
	\end{equation*}
	provided that $M$ is sufficiently large. Additionally, computation gives
	\begin{equation*}
		\begin{aligned}
			L(\varphi)&=M\left(\sqrt{\omega}\beta(1+\beta)\psi^{\beta-1}-\frac{1}{2}\left(\psi-(1+\beta)\psi^{1+\beta}\right)+\frac{\psi \omega_\psi}{4\omega}\psi \left(1-\psi^\beta \right)\right)\\
			&\geq \delta M \psi^{\beta-\frac{1}{2}}
		\end{aligned}
	\end{equation*}
	where $\delta$ is a small positive constant. Then one has
	\begin{equation*}
		L(\varphi)\pm L(r)\geq \delta M \psi^{\beta-\frac{1}{2}} \pm \left|\frac{r^2}{2\omega}\right| \geq 0
	\end{equation*}
	for $M$ large enough. Application of the maximum principle yields
	\begin{equation*}
		\left|\omega_\xi\right| \lesssim \psi \quad \textit{for}\quad 0\leq \psi \leq \psi_0.
	\end{equation*}
	By Lemma \ref{barw} ang Lemma \ref{lemma1}, one can obtain
	\begin{equation*}
		\left| -\frac{\omega_\xi}{2\omega}\right|\lesssim 1 \quad \textit{for}\quad 0\leq \psi \leq \psi_0.
	\end{equation*}

\end{proof}

\section{Proof of Theorem \ref{energy1}}

Let $\tilde{\omega}=\omega-\bar{\omega}$, where 
\begin{equation}
	\bar{\omega}\left(\int_{0}^{y}\frac{u(x,y')}{\sqrt{x+1}}dy'\right)=\left[f'\left(\frac{y}{\sqrt{2(x+1)}}\right)\right]^2.
\end{equation}
This perturbation satisfies the governing equation:
\begin{equation}\label{error}
	\tilde{\omega}_\xi-\frac{1}{2}\psi\tilde{\omega}_\psi-\sqrt{\bar{\omega}}\tilde{\omega}_{\psi\psi}-\frac{\bar{\omega}_{\psi\psi}}{2\sqrt{\bar{\omega}}}\tilde{\omega}=F,
\end{equation}
where the nonlinear term $F$ is given by:
\begin{equation}\label{nonline}
	F=\frac{\tilde{\omega}\tilde{\omega}_{\psi\psi}}{\sqrt{\bar{\omega}}+\sqrt{\omega}}-\frac{\bar{\omega}_{\psi\psi}}{2\sqrt{\bar{\omega}}}\frac{\tilde{\omega}^2}{\left(\sqrt{\bar{\omega}}+\sqrt{\omega}\right)^2},
\end{equation}
with boundary conditions
\begin{equation*}
	\tilde{\omega}(\xi,0)=\tilde{\omega}(\xi,\infty)=0, \quad \tilde{\omega}_0 (\psi)=\omega_0(\psi)-\bar{\omega}(\psi).
\end{equation*}

We define the weight function $\rho(\psi)$ through
\begin{equation}\label{wei}
	A \tilde{\omega}_\xi-(\rho \tilde{\omega}_\psi)_{\psi}-\frac{\bar{\omega}_{\psi\psi}}{2\sqrt{\bar{\omega}}} A \tilde{\omega}=AF,
\end{equation}
where $A=\frac{\rho}{\sqrt{\bar{\omega}}}$ and $\rho=exp\left(\int_{0}^{\psi}\frac{s}{2\sqrt{\bar{\omega}(s)}}ds\right)$.

\subsection{Stability Analysis}
\ 
\newline
\indent Define the following weighted energy norms:
\begin{align*}
	\|\tilde{\omega}\|_{X_0}^2&=\sup_{\xi \geq d}\left\|\sqrt{A} \tilde{\omega} e^\xi\right\|_{L_\psi^2}^2+\left\|\sqrt{\rho} \tilde{\omega}_\psi e^{\frac{3}{4} \xi}\right\|_{L_{\xi,\psi}^2}^2;\\
	\|\tilde{\omega}\|_{X_1}^2&=\sup _{\xi \geq d} \left\|\sqrt{\rho} \tilde{\omega}_\psi e^{\frac{3}{4}\xi}\right\|_{L_\psi^2}^2+\left\|\sqrt{A} \tilde{\omega}_{\xi} e^{\frac{3}{4}\xi}\right\|_{L_{\xi,\psi}^2}^2;\\
	\|\tilde{\omega}\|_{X_2}^2&=\sup _{\xi \geq d}\left\|\sqrt{A} \tilde{\omega}_\xi e^{\frac{3}{4}\xi}\right\|_{L_\psi^2}^2+\left\|\sqrt{\rho} \tilde{\omega}_{\xi\psi} e^{\frac{3}{4}\xi}\right\|_{L_{\xi,\psi}^2}^2+\sup _{\xi \geq d}\left\|\frac{\sqrt{A} \sqrt{\bar{\omega}}}{1+\psi} \tilde{\omega}_{\psi\psi} e^{\frac{3}{4}\xi}\right\|_{L_\psi^2}^2;\\
	\|\tilde{\omega}\|_{X_3}^2&=\sup _{\xi \geq d}\left\|\sqrt{\rho} \tilde{w}_{\xi\psi} e^{\frac{3}{4} \xi}\right\|_{L_\psi^2}^2+\left\|\sqrt{A} \tilde{\omega}_{\xi\xi} e^{\frac{3}{4} \xi}\right\|_{L_{\xi,\psi}^2}^2+\left\|\frac{\sqrt{A} \sqrt{\omega}}{1+\psi} \tilde{\omega}_{\xi\psi\psi} e^{\frac{3}{4}\xi}\right\|_{L_{\xi, \psi}^2}^2.
\end{align*}
Construct the solution space:
\begin{equation}\label{space}
	\|\tilde{\omega}\|_{\mathbb{X} }=\left\{\tilde{\omega}\left|\sum_{i=0}^{3}\|\tilde{\omega}\|_{X_i}<\infty \right.\right\}.
\end{equation}
Define corresponding norms for initial data:
\begin{equation}\label{inspace}
	\|\tilde{\omega}_0\|_{\mathbb{X}_{0}}^2 =\left\| \sqrt{A} \tilde{\omega}_0\right\|_{L^2_\psi}^2 + \left\| \sqrt{\rho} \tilde{\omega}_{0\psi}\right\|_{L^2_\psi}^2 +\left\| \sqrt{A} \tilde{\omega}_{0\xi}\right\|_{L^2_\psi}^2+ \left\| \sqrt{\rho} \tilde{\omega}_{0\xi\psi}\right\|_{L^2_\psi}^2.
\end{equation}

This analytical framework yields the following stability theorem, which constitutes a low-regularity version of Theorem \ref{energy1} with relaxed decay requirements.
\begin{theorem}\label{sta}
	Under the assumptions in Theorem \ref{Oleinik}, if the initial data $u_0(y)$ satisfies
	\begin{equation}\label{assume1}
		\sum_{j=0}^{3}\left\|e^{\frac{y^2}{8}}\partial_y^j (u_0-\bar{u}_0)\right\|_{L^2_y} \leq \kappa \ll 1.
	\end{equation}
	Then the perturbation $\tilde{\omega}=\omega-\bar{\omega}$ exhibits exponential decay:
	\begin{align*}
		&\left\|\sqrt{A} \tilde{\omega} \right\|_{L_\psi^2} \lesssim e^{-\xi}, \quad \left\|\sqrt{\rho} \tilde{\omega}_\psi \right\|_{L_\psi^2} \lesssim e^{-\frac{3}{4}\xi},\\
		&\left\|\sqrt{A} \tilde{\omega}_\xi\right\|_{L_\psi^2} \lesssim e^{-\frac{3}{4}\xi},\quad \left\|\sqrt{\rho} \tilde{\omega}_{\xi\psi} \right\|_{L_\psi^2} \lesssim e^{-\frac{3}{4}\xi},\\
		&\left\|\frac{\sqrt{A} \sqrt{\bar{\omega}}}{1+\psi} \tilde{\omega}_{\psi\psi}\right\|_{L_\psi^2} \lesssim e^{-\frac{3}{4}\xi}.
	\end{align*}
\end{theorem}

The proof of Theorem \ref{sta} follows from Lemmas \ref{ini} and \ref{exi}. First we establish the initial data compatibility.
\begin{lemma}\label{ini}
	Under assumption (\ref{assume1}), the transformed initial data $\tilde{\omega}_0(\psi)$ satisfies:
	\begin{equation*}
		\|\tilde{\omega}_0\|_{\mathbb{X}_{0}} \lesssim \kappa \ll 1.
	\end{equation*}
\end{lemma}
\begin{proof}
	The modified von Mises transformation yields
	\begin{equation*}
		\psi=\int_{0}^{y} \frac{u(x,y')}{\sqrt{x+1}} dy',
	\end{equation*}
	with asymptotic behavior $\psi \sim y/\sqrt{x+1}$ as $y\rightarrow \infty$. Consequently,
	\begin{equation*}
		\rho(\psi) \sim e^\frac{\psi^2}{4} \sim e^\frac{y^2}{4(x+1)}, \quad \text{as}\quad y\rightarrow \infty.
	\end{equation*}

    The key estimates proceed as follows:
    
    \textit{Perturbation magnitude:}
	\begin{align*}
		\int_{0}^{\infty} A(\psi) \tilde{\omega}_0^2 d\psi &= \int_{0}^{\infty} A(y) \left(u_0^2-\bar{u}_0^2\right)^2\frac{u_0}{\sqrt{d}} d y\\
		& \lesssim \int_{0}^{\infty} \rho(y) \left(u_0-\bar{u}_0\right)^2 \left(u_0+\bar{u}_0 \right)^2 d y \\
		&\lesssim \int_{0}^{\infty} e^{\frac{y^2}{4}} \left(u_0-\bar{u}_0\right)^2 d y \\
		&\lesssim \kappa^2 .
	\end{align*}
	
	\textit{$\psi$-derivative:}
	\begin{align*}
		\int_{0}^{\infty} \rho(\psi) \tilde{\omega}_{0\psi}^2 d\psi &= \int_{0}^{\infty} \rho(y) 4d \left(u_{0y}-\bar{u}_{0y}\right)^2 \frac{u_0}{\sqrt{d}} d y\\
		& \lesssim \int_{0}^{\infty} \rho(y)\left(u_{0y}-\bar{u}_{0y}\right)^2 d y\\
		&\lesssim \int_{0}^{\infty} e^{\frac{y^2}{4}} \left(u_{0y}-\bar{u}_{0y}\right)^2 d y\\
		&\lesssim \kappa^2 .
	\end{align*}

	\textit{$\xi$-derivative:}
	\begin{align*}
		\int_{0}^{\infty} A(\psi) \tilde{\omega}_{0\xi}^2 d\psi &= \int_{0}^{\infty} A(\psi) \left(\frac{\psi}{2}\tilde{\omega}_{0\psi}+\left(\sqrt{\omega_0}\omega_{0\psi\psi}-\sqrt{\bar{\omega}_0}\bar{\omega}_{0\psi\psi}\right)\right)^2 d\psi\\
		&\lesssim \int_{0}^{\infty} A(y) \psi^2 \left(u_{0y}-\bar{u}_{0y}\right)^2 \frac{u_0}{\sqrt{d}} dy + \int_{0}^{\infty} A(y) \left(u_{0yy}-\bar{u}_{0yy}\right)^2 \frac{u_0}{\sqrt{d}} dy\\
		&\lesssim \int_{0}^{\infty} e^{\frac{y^2}{4}} \left(u_{0y}-\bar{u}_{0y}\right)^2 dy+\int_{0}^{\infty} e^{\frac{y^2}{4}} \left(u_{0yy}-\bar{u}_{0yy}\right)^2 dy\\
		&\lesssim \kappa^2 .
	\end{align*}

    \textit{Mixed derivative:}
	\begin{align*}
		\int_{0}^{\infty} \rho(\psi) \tilde{\omega}_{0\xi\psi}^2 d\psi &= \int_{0}^{\infty} \rho(\psi) \left(\frac{\psi}{2}\tilde{\omega}_{0\psi}+\left(\sqrt{\omega_0}\omega_{0\psi\psi}-\sqrt{\bar{\omega}_0}\bar{\omega}_{0\psi\psi}\right)\right)_\psi^2 d\psi\\
		&\lesssim \int_{0}^{\infty} \rho(y)\left(u_{0y}-\bar{u}_{0y}\right)^2 \frac{u_0}{\sqrt{d}}dy + \int_{0}^{\infty} \rho(y) \psi^2 \left(u_{0yy}-\bar{u}_{0yy}\right)^2 \frac{\sqrt{d}}{u_0} dy\\
		&+ \int_{0}^{\infty} \rho(y) \left(u_{0yyy}-\bar{u}_{0yyy}\right)^2 \frac{\sqrt{d}}{u_0} dy\\
		&\lesssim \int_{0}^{\infty} e^{\frac{y^2}{4}} \left(u_{0y}-\bar{u}_{0y}\right)^2 dy+\int_{0}^{\infty} e^{\frac{y^2}{4}} \left(u_{0yy}-\bar{u}_{0yy}\right)^2 dy\\
		&+\int_{0}^{\infty} e^{\frac{y^2}{4}} \frac{\left(u_{0yyy}-\bar{u}_{0yyy}\right)^2}{\bar{u}_0} dy\\
		&\lesssim \kappa^2.
	\end{align*}
	The final inequality follows from application of the compatibility condition \eqref{compatibility} in Theorem~\ref{Oleinik}.
\end{proof}

\begin{lemma}\label{exi}
	There exists $\kappa \in (0,1)$ sufficiently small such that under Theorem \ref{sta}'s assumptions, the perturbation $\tilde{\omega}$ exists globally in space $\|\cdot\|_{\mathbb{X}}$.
\end{lemma}
\begin{proof}
    The proof proceeds through five key steps:
	\begin{enumerate}[Step 1:]      
		\item \textbf{Estimation of $\|\tilde{\omega}\|_{X_0}$.} Applying the energy estimate with test function $\tilde{\omega} e^{2\xi}$ to equation \eqref{wei}, one derives through integration by parts:
		\begin{align*}
			&\frac{1}{2} \frac{\partial}{\partial \xi}\left(e^{2\xi} \int_0^{\infty} A \tilde{\omega}^2 d \psi\right)-e^{2\xi} \int_0^{\infty} A \tilde{\omega}^2 d \psi\\
			+&e^{2 \xi} \int_0^{\infty} \rho \tilde{\omega}_\psi^2 d \psi-e^{2\xi}\int_0^{\infty} \frac{\bar{\omega}_{\psi\psi}}{2\sqrt{\bar{\omega}}} A \tilde{\omega}^2 d \psi= e^{2 \xi} \int_0^{\infty} A F \tilde{w} d \psi.
		\end{align*}
		The proof of Proposition \ref{eigenvalue} in Appendix \ref{app:eigenvalue} yields:
		\begin{equation*}
			\int_0^{\infty} \rho \tilde{\omega}_\psi^2 d \psi-\int_0^{\infty} \frac{\bar{\omega}_{\psi \psi}}{2 \sqrt{\bar{\omega}}} A \tilde{\omega}^2 d \psi \geq \int_0^{\infty} A \tilde{\omega}^2 d \psi.
		\end{equation*}
		This leads to the following inequality:
		\begin{equation*}
			\frac{1}{2} \frac{\partial}{\partial \xi}\left(e^{2 \xi} \int_0^{\infty} A \tilde{\omega}^2 d \psi\right) \leq \left(e^{2 \xi} \int_0^{\infty} A \tilde{\omega}^2 d \psi\right)^{\frac{1}{2}}\left(e^{2\xi} \int_0^{\infty} A F^2 d \psi\right)^{\frac{1}{2}}.
		\end{equation*}
		
		Applying Grönwall's inequality yield the first key estimate:
		\begin{equation}\label{X01}
			\sup _{\xi \geq d}\left\|\sqrt{A} \tilde{\omega} e^\xi \right\|_{L_\psi^2}^2 \lesssim \left\|\sqrt{A} \tilde{\omega}_0\right\|_{L_\psi^2}^2+ \left(\int_0^\infty e^{\xi} \left\|\sqrt{A} F \right\|_{L_{\psi}^2} d\xi\right)^2.
		\end{equation}

		Using the alternative test function $\tilde{\omega}e^{\frac{3}{2}\xi}$ produces
		\begin{align*}
			&\frac{1}{2} \frac{\partial}{\partial \xi}\left(e^{\frac{3}{2}\xi} \int_0^{\infty} A \tilde{\omega}^2 d \psi\right)-\frac{3}{4}e^{\frac{3}{2}\xi} \int_0^{\infty} A \tilde{\omega}^2 d \psi\\
			+&e^{\frac{3}{2} \xi} \int_0^{\infty} \rho \tilde{\omega}_\psi^2 d \psi-e^{\frac{3}{2}\xi}\int_0^{\infty} \frac{\bar{\omega}_{\psi\psi}}{2\sqrt{\bar{\omega}}} A \tilde{\omega}^2 d \psi= e^{\frac{3}{2} \xi} \int_0^{\infty} A F \tilde{w} d \psi,
		\end{align*}
		then one has
		\begin{align*}
			&\frac{1}{2} \frac{\partial}{\partial \xi}\left(e^{\frac{3}{2}\xi} \int_0^{\infty} A \tilde{\omega}^2 d \psi\right)+e^{\frac{3}{2} \xi} \int_0^{\infty} \rho \tilde{\omega}_\psi^2 d \psi \\
			\lesssim & e^{\frac{3}{2}\xi} \int_0^{\infty} A \tilde{\omega}^2 d \psi+e^{\frac{3}{2} \xi} \int_0^{\infty} A F^2 d \psi.
		\end{align*}	
		Integration over $\xi$ gives the secondary estimate:
		\begin{equation}\label{X02}
			\left\|\sqrt{\rho} \tilde{\omega}_\psi e^{\frac{3}{4} \xi}\right\|_{L_{\xi,\psi}^2}^2 \lesssim \left\|\sqrt{A} \tilde{\omega}_0\right\|_{L_\psi^2}^2+\sup_{\xi \geq d}\left\|\sqrt{A} \tilde{\omega} e^\xi\right\|_{L_\psi^2}^2+\left\|\sqrt{A} F e^{\frac{3}{4} \xi}\right\|_{L_{\xi,\psi}^2}^2.
		\end{equation}

		Concentrating on the term $\left\|\sqrt{A} F\right\|_{L_{\psi}^2}$, one knows
		\begin{align*}
			\left\|\sqrt{A} F\right\|_{L_{\psi}^2}^2 &\lesssim \int_{0}^{\infty}\frac{A\tilde{\omega}^2\tilde{\omega}_{\psi\psi}^2}{\bar{\omega}}d\psi+\int_{0}^{\infty}\frac{A\tilde{\omega}^4}{\bar{\omega}^2}\\
			&=:I_1+I_2.
		\end{align*}
		On the one hand, for $\psi\leq 1$, one has
		\begin{align*}
			I_1&\lesssim \left\|\frac{\tilde{\omega}}{\bar{\omega}}\right\|^2_{L_\psi^\infty} \left\|\sqrt{A}\sqrt{\bar{\omega}}\tilde{\omega}_{\psi\psi}\right\|^2_{L_\psi^2}\lesssim \left\|\tilde{\omega}_{\psi}\right\|^2_{L_\psi^\infty} \left\|\frac{\sqrt{A}\sqrt{\bar{\omega}}}{1+\psi}\tilde{\omega}_{\psi\psi}\right\|^2_{L_\psi^2}\\ 
			&\lesssim \left\|\frac{\sqrt{A}\sqrt{\bar{\omega}}}{1+\psi}\tilde{\omega}_{\psi\psi}\right\|^4_{L_\psi^2} \lesssim e^{-3\xi} \|\tilde{\omega}\|_{X_2}^4;\\
			I_2&\lesssim \left\|\frac{\tilde{\omega}}{\bar{\omega}}\right\|^2_{L_\psi^\infty} \left\|\sqrt{A}\tilde{\omega}\right\|^2_{L_\psi^2}\lesssim \left\|\tilde{\omega}_{\psi}\right\|^2_{L_\psi^\infty}
			\left\|\sqrt{A}\tilde{\omega}\right\|^2_{L_\psi^2}\\
			&\lesssim \left\|\frac{\sqrt{A}\sqrt{\bar{\omega}}}{1+\psi}\tilde{\omega}_{\psi\psi}\right\|^2_{L_\psi^2}\left\|\sqrt{A}\tilde{\omega}\right\|^2_{L_\psi^2}\lesssim e^{-\frac{7}{2}\xi} \|\tilde{\omega}\|_{X_0}^2\|\tilde{\omega}\|_{X_2}^2.
		\end{align*}		
		On the other hand, for $\psi\geq 1$, one has
		\begin{align*}
			I_1&\lesssim \left\|\sqrt[4]{A}\tilde{\omega}\right\|^2_{L_\psi^\infty} \int_{0}^{\infty}\sqrt{A}\tilde{\omega}_{\psi\psi} d\psi\\
			&\lesssim \left( \left\|\sqrt{A}\tilde{\omega}\right\|_{L_\psi^2}\left\|\sqrt{\rho}\tilde{\omega}_\psi\right\|_{L_\psi^2}+\left\|\sqrt{A}\tilde{\omega}\right\|_{L_\psi^2}^2\right) \left\|\frac{\sqrt{A}\sqrt{\bar{\omega}}}{1+\psi}\tilde{\omega}_{\psi\psi}\right\|^2_{L_\psi^2}\\
			&\lesssim e^{-\frac{7}{2}\xi} \|\tilde{\omega}\|_{X_0}^2\|\tilde{\omega}\|_{X_2}^2 + e^{-\frac{13}{4}\xi} \|\tilde{\omega}\|_{X_0} \|\tilde{\omega}\|_{X_1}\|\tilde{\omega}\|_{X_2}^2;\\
			I_2&\lesssim \left\|\tilde{\omega}\right\|^2_{L_\psi^\infty}\left\|\sqrt{A}\tilde{\omega}\right\|^2_{L_\psi^2} \lesssim \left\|\sqrt{A}\tilde{\omega}\right\|^3_{L_\psi^2} \left\|\sqrt{\rho}\tilde{\omega}_{\psi}\right\|_{L_\psi^2}\\
			&\lesssim e^{-\frac{15}{4}\xi} \|\tilde{\omega}\|_{X_0}^3 \|\tilde{\omega}\|_{X_1}.
		\end{align*}
		
		Combining these estimates establishes the bounds:
		\begin{equation}\label{AF}
			\begin{aligned}
				&\left(\int_0^\infty e^{\xi} \left\|\sqrt{A} F \right\|_{L_{\psi}^2} d\xi\right)^2  \lesssim \|\tilde{\omega}\|_{\mathbb{X} }^4,\\
				&\left\|\sqrt{A} F e^{\frac{3}{4} \xi}\right\|_{L_{\xi,\psi}^2}^2  \lesssim \|\tilde{\omega}\|_{\mathbb{X} }^4.
			\end{aligned}
		\end{equation}

        The combination of \eqref{space}, \eqref{inspace}, \eqref{X01}, \eqref{X02} and \eqref{AF} yields the fundamental estimate:
		\begin{equation}\label{X0}
			\|\tilde{\omega}\|_{X_0}^2 \lesssim \left\|\sqrt{A} \tilde{\omega}_0\right\|_{L_\psi^2}^2 +\|\tilde{\omega}\|_{\mathbb{X} }^4.
		\end{equation}

		\item \textbf{Estimation of $\|\tilde{\omega}\|_{X_1}$.} Applying the energy estimate with test function $\tilde{\omega}_\xi e^{\frac{3}{2}\xi}$ to equation \eqref{wei}, one obtains through integration by parts:
		\begin{align*}
			&e^{\frac{3}{2} \xi} \int_0^{\infty} A \tilde{\omega}_\xi^2 d \psi+\frac{1}{2} \frac{\partial}{\partial \xi}\left(e^{\frac{3}{2} \xi} \int_0^{\infty} \rho \tilde{\omega}_\psi^2 d \psi\right)-\frac{3}{4} e^{\frac{3}{2} \xi} \int_0^{\infty} \rho \tilde{\omega}_\psi^2 d \psi\\
			-&e^{\frac{3}{2}\xi} \int_0^{\infty} \frac{\bar{\omega}_{\psi \psi}}{2 \sqrt{\bar{\omega}}} A \tilde{\omega} \tilde{\omega}_\xi d \psi =e^{\frac{3}{2}\xi} \int_0^{\infty} A F \tilde{\omega}_\xi d \psi.
		\end{align*}
		This generates the differential inequality:
		\begin{align*}
			&e^{\frac{3}{2} \xi} \int_0^{\infty} A \tilde{\omega}_\xi^2 d \psi+\frac{1}{2} \frac{\partial}{\partial \xi}\left(e^{\frac{3}{2} \xi} \int_0^{\infty} \rho \tilde{\omega}_\psi^2 d \psi\right)\\
			\lesssim & e^{\frac{3}{2} \xi} \int_0^{\infty} \rho \tilde{\omega}_\psi^2 d \psi +e^{\frac{3}{2} \xi} \int_0^{\infty} A \tilde{\omega}^2 d \psi+e^{\frac{3}{2}\xi} \int_0^{\infty} A F^2 d \psi.
		\end{align*}
		Integrating with respect to $\xi$ yields
		\begin{equation}\label{X11}
			\|\tilde{\omega}\|_{X_1}^2 \lesssim \left\|\sqrt{\rho} \tilde{\omega}_{0\psi}\right\|_{L_\psi^2}^2 + \|\tilde{\omega}\|_{X_0}^2 +\left\|\sqrt{A} F e^{\frac{3}{4} \xi}\right\|_{L_{\xi,\psi}^2}^2.
		\end{equation}
		Combining with previous estimates \eqref{AF} and \eqref{X0}, one has
		\begin{equation}\label{X1}
			\|\tilde{\omega}\|_{X_1}^2 \lesssim \|\tilde{\omega}_0\|_{\mathbb{X}_{0}} + \|\tilde{\omega}\|_{\mathbb{X} }^4.
		\end{equation}

		\item \textbf{Estimation of $\|\tilde{\omega}\|_{X_2}$.} Define the derivative quantity $r := \tilde{\omega}_\xi$ satisfying the following equation:
		\begin{equation}\label{r1}
			r_\xi-\frac{1}{2}\psi r_\psi-\sqrt{\bar{\omega}} r_{\psi\psi}-\frac{\bar{\omega}_{\psi\psi}}{2\sqrt{\bar{\omega}}} r =F_\xi.
		\end{equation}
		This yields the weighted equation:
		\begin{equation}\label{r}
			A r_\xi-(\rho r_\psi)_{\psi}-\frac{\bar{\omega}_{\psi\psi}}{\sqrt{\bar{\omega}}+\sqrt{\omega}} A r=AF_\xi.
		\end{equation}
		
		Applying the energy estimate with test function $r e^{\frac{3}{2}\xi}$ to equation \eqref{r} produces:
		\begin{align*}
			&\frac{1}{2} \frac{\partial}{\partial \xi}\left(e^{\frac{3}{2} \xi} \int_0^{\infty} A r^2 d \psi\right)-\frac{3}{4} e^{\frac{3}{2} \xi} \int_0^{\infty} A r^2 d \psi + e^{\frac{3}{2} \xi} \int_0^{\infty} \rho r_\psi^2 d \psi\\
			- & e^{\frac{3}{2}\xi} \int_0^{\infty} \frac{\bar{\omega}_{\psi \psi}}{2 \sqrt{\bar{\omega}}} A r^2 d \psi =e^{\frac{3}{2}\xi} \int_0^{\infty} A F_\xi r d \psi.
		\end{align*}
		This leads to the following inequality:
		\begin{align*}
			&\frac{1}{2} \frac{\partial}{\partial \xi}\left(e^{\frac{3}{2} \xi} \int_0^{\infty} A r^2 d \psi\right)+ e^{\frac{3}{2} \xi} \int_0^{\infty} \rho r_\psi^2 d \psi\\
			\lesssim & e^{\frac{3}{2} \xi} \int_0^{\infty} A r^2 d \psi +e^{\frac{3}{2}\xi} \int_0^{\infty} A F_\xi^2 d \psi.
		\end{align*}
		Integrating with respect to $\xi$, one can obtain
		\begin{equation}\label{X21}
			\begin{aligned}
				&\sup _{\xi \geq d}\left\|\sqrt{A} \tilde{\omega}_\xi e^{\frac{3}{4}\xi}\right\|_{L_\psi^2}^2+\left\|\sqrt{\rho} \tilde{\omega}_{\xi\psi} e^{\frac{3}{4}\xi}\right\|_{L_{\xi,\psi}^2}^2\\
				\lesssim & \left\|\sqrt{A} \tilde{\omega}_{0\xi}\right\|_{L_\psi^2}^2 + \|\tilde{\omega}\|_{X_1}^2 +\left\|\sqrt{A} F_\xi e^{\frac{3}{4} \xi}\right\|_{L_{\xi,\psi}^2}^2.
			\end{aligned}	
		\end{equation}

		Focusing on the nonlinear term $\left\|\sqrt{A} F_\xi\right\|_{L_{\psi}^2}$, one decomposes
		\begin{align*}
			&\left\|\sqrt{A} F_\xi \right\|_{L_{\psi}^2}^2 \\
			\lesssim & \int_{0}^{\infty}\frac{A\tilde{\omega}_\xi^2\tilde{\omega}_{\psi\psi}^2}{\bar{\omega}}d\psi+\int_{0}^{\infty}\frac{A\tilde{\omega}^2\tilde{\omega}_{\xi\psi\psi}^2}{\bar{\omega}}d\psi + \int_{0}^{\infty}\frac{A\tilde{\omega}^2\tilde{\omega}_{\psi\psi}^2 \tilde{\omega}_\xi^2 }{\bar{\omega}^3}d\psi\\
			&+\int_{0}^{\infty}\frac{A\tilde{\omega}^2\tilde{\omega}_\xi^2}{\bar{\omega}^2} + \int_{0}^{\infty}\frac{A\tilde{\omega}^4 \tilde{\omega}_\xi^2}{\bar{\omega}^4}\\
			=: & J_1+J_2+J_3+J_4+J_5.
		\end{align*}
		On the one hand, for $\psi\leq 1$, one has
		\begin{align*}
			J_1+J_2&\lesssim \left\|\frac{\tilde{\omega}_\xi}{\bar{\omega}}\right\|^2_{L_\psi^\infty} \left\|\sqrt{A}\sqrt{\bar{\omega}}\tilde{\omega}_{\psi\psi}\right\|^2_{L_\psi^2} + \left\|\frac{\tilde{\omega}}{\bar{\omega}}\right\|^2_{L_\psi^\infty} \left\|\sqrt{A}\sqrt{\bar{\omega}}\tilde{\omega}_{\xi\psi\psi}\right\|^2_{L_\psi^2}\\
		    &\lesssim \left\|\frac{\sqrt{A}\sqrt{\bar{\omega}}}{1+\psi}\tilde{\omega}_{\xi\psi\psi}\right\|^2_{L_\psi^2} \left\|\frac{\sqrt{A}\sqrt{\bar{\omega}}}{1+\psi}\tilde{\omega}_{\psi\psi}\right\|^2_{L_\psi^2}\\
		    &\lesssim e^{-\frac{3}{2}\xi} \left\|\frac{\sqrt{A}\sqrt{\bar{\omega}}}{1+\psi}\tilde{\omega}_{\xi\psi\psi}\right\|^2_{L_\psi^2} \|\tilde{\omega}\|_{X_2}^2;\\
			J_3&\lesssim \left\|\frac{\tilde{\omega}}{\bar{\omega}}\right\|^2_{L_\psi^\infty} \left\|\frac{\tilde{\omega}_\xi}{\bar{\omega}}\right\|^2_{L_\psi^\infty} \left\|\sqrt{A}\sqrt{\bar{\omega}}\tilde{\omega}_{\psi\psi}\right\|^2_{L_\psi^2}\\
		    &\lesssim \left\|\frac{\sqrt{A}\sqrt{\bar{\omega}}}{1+\psi}\tilde{\omega}_{\xi\psi\psi}\right\|^2_{L_\psi^2} \left\|\frac{\sqrt{A}\sqrt{\bar{\omega}}}{1+\psi}\tilde{\omega}_{\psi\psi}\right\|^4_{L_\psi^2}\\
		    &\lesssim e^{-3\xi} \left\|\frac{\sqrt{A}\sqrt{\bar{\omega}}}{1+\psi}\tilde{\omega}_{\xi\psi\psi}\right\|^2_{L_\psi^2} \|\tilde{\omega}\|_{X_2}^4;\\
			J_4+J_5&\lesssim \left\|\frac{\tilde{\omega}}{\bar{\omega}}\right\|^2_{L_\psi^\infty} \left\|\sqrt{A}\tilde{\omega}_\xi\right\|^2_{L_\psi^2}+\left\|\frac{\tilde{\omega}}{\bar{\omega}}\right\|^4_{L_\psi^\infty} \left\|\sqrt{A}\tilde{\omega}_\xi\right\|^2_{L_\psi^2}\\
			&\lesssim \left\|\frac{\sqrt{A}\sqrt{\bar{\omega}}}{1+\psi}\tilde{\omega}_{\psi\psi}\right\|^2_{L_\psi^2}\left\|\sqrt{A}\tilde{\omega}_\xi\right\|^2_{L_\psi^2}+\left\|\frac{\sqrt{A}\sqrt{\bar{\omega}}}{1+\psi}\tilde{\omega}_{\psi\psi}\right\|^4_{L_\psi^2}\left\|\sqrt{A}\tilde{\omega}_\xi\right\|^2_{L_\psi^2}\\
		    &\lesssim e^{-3\xi} \|\tilde{\omega}\|_{X_0}^2\|\tilde{\omega}\|_{X_2}^2+e^{-\frac{9}{2}\xi} \|\tilde{\omega}\|_{X_0}^2\|\tilde{\omega}\|_{X_2}^4.
		\end{align*}		

		On the other hand, for $\psi\geq 1$, one has
		\begin{align*}
			J_1 &\lesssim \left\|\sqrt[4]{A}\tilde{\omega}_\xi\right\|^2_{L_\psi^\infty} \left\|\frac{\sqrt{A}\sqrt{\bar{\omega}}}{1+\psi}\tilde{\omega}_{\psi\psi}\right\|^2_{L_\psi^2}\\
		    &\lesssim \left( \left\|\sqrt{A}\tilde{\omega}_\xi\right\|_{L_\psi^2}\left\|\sqrt{\rho}\tilde{\omega}_{\xi\psi}\right\|_{L_\psi^2}+\left\|\sqrt{A}\tilde{\omega}_{\xi}\right\|_{L_\psi^2}^2\right) \left\|\frac{\sqrt{A}\sqrt{\bar{\omega}}}{1+\psi}\tilde{\omega}_{\psi\psi}\right\|^2_{L_\psi^2}\\
		    &\lesssim e^{-3\xi} \|\tilde{\omega}\|_{X_2}^3\|\tilde{\omega}\|_{X_3} + e^{-3\xi} \|\tilde{\omega}\|_{X_2}^4;\\
			J_2 &\lesssim \left\|\sqrt[4]{A}\tilde{\omega}\right\|^2_{L_\psi^\infty}\left\|\frac{\sqrt{A}\sqrt{\bar{\omega}}}{1+\psi}\tilde{\omega}_{\xi\psi\psi}\right\|^2_{L_\psi^2}\\
		    &\lesssim e^{-2\xi} \left\|\frac{\sqrt{A}\sqrt{\bar{\omega}}}{1+\psi}\tilde{\omega}_{\xi\psi\psi}\right\|^2_{L_\psi^2} \|\tilde{\omega}\|_{X_0}^2\|+ e^{-\frac{7}{4}\xi} \left\|\frac{\sqrt{A}\sqrt{\bar{\omega}}}{1+\psi}\tilde{\omega}_{\xi\psi\psi}\right\|^2_{L_\psi^2} \|\tilde{\omega}\|_{X_0} \|\tilde{\omega}\|_{X_1};\\
			J_3 &\lesssim \left\|\sqrt[4]{A}\tilde{\omega}\right\|^2_{L_\psi^\infty} \left\|\tilde{\omega}_{\xi}\right\|^2_{L_\psi^\infty} \left\|\frac{\sqrt{A}\sqrt{\bar{\omega}}}{1+\psi}\tilde{\omega}_{\psi\psi}\right\|^2_{L_\psi^2}\\
		    & \lesssim \left(e^{-2\xi} \|\tilde{\omega}\|_{X_0}^2\|+ e^{-\frac{7}{4}\xi}  \|\tilde{\omega}\|_{X_0} \|\tilde{\omega}\|_{X_1}\right)\left(e^{-3\xi} \|\tilde{\omega}\|_{X_2}^3\|\tilde{\omega}\|_{X_3} + e^{-3\xi} \|\tilde{\omega}\|_{X_2}^4\right);\\
			J_4&\lesssim \left\|\tilde{\omega}\right\|^2_{L_\psi^\infty} \left\|\sqrt{A}\tilde{\omega}_{\xi} \right\|^2_{L_\psi^2} \lesssim e^{-\frac{13}{4}\xi} \|\tilde{\omega}\|_{X_0} \|\tilde{\omega}\|_{X_1}\|\tilde{\omega}\|_{X_2}^2;\\
			J_5&\lesssim \left\|\tilde{\omega}\right\|^4_{L_\psi^\infty} \left\|\sqrt{A}\tilde{\omega}_{\xi} \right\|^2_{L_\psi^2} \lesssim e^{-5\xi} \|\tilde{\omega}\|_{X_0}^2 \|\tilde{\omega}\|_{X_1}^2\|\tilde{\omega}\|_{X_2}^2.
		\end{align*}
		
		Synthesizing these estimates yields the bound:
		\begin{equation}\label{AFX}
				\left\|\sqrt{A} F_\xi e^{\frac{3}{4} \xi}\right\|_{L_{\xi,\psi}^2}^2 \lesssim \|\tilde{\omega}\|_{\mathbb{X} }^4 +\|\tilde{\omega}\|_{\mathbb{X} }^6.
		\end{equation}

		From the error equation \eqref{error}, one obtains:
		\begin{equation}\label{X22}
			\begin{aligned}
				&\sup _{\xi \geq d}\left\|\frac{\sqrt{A} \sqrt{\bar{\omega}}}{1+\psi} \tilde{\omega}_{\psi\psi} e^{\frac{3}{4}\xi}\right\|_{L_\psi^2}^2\\
				\lesssim & \sup _{\xi \geq d}\left\|\sqrt{A} \tilde{\omega}_{\xi} e^{\frac{3}{4}\xi}\right\|_{L_\psi^2}^2+\sup _{\xi \geq d}\left\|\sqrt{\rho} \tilde{\omega}_{\psi} e^{\frac{3}{4}\xi}\right\|_{L_\psi^2}^2+\sup _{\xi \geq d}\left\|\sqrt{A} \tilde{\omega} e^{\frac{3}{4}\xi}\right\|_{L_\psi^2}^2+\sup _{\xi \geq d}\left\|\sqrt{A} F e^{\frac{3}{4}\xi}\right\|_{L_\psi^2}^2\\
				\lesssim & \sup _{\xi \geq d}\left\|\sqrt{A} \tilde{\omega}_{\xi} e^{\frac{3}{4}\xi}\right\|_{L_\psi^2}^2 + \|\tilde{\omega}\|_{\mathbb{X} }^4. 
			\end{aligned}
		\end{equation}

		By combining (\ref{X21}), (\ref{X22}), (\ref{X1}) and (\ref{AFX}), one completes the higher-order estimate:
		\begin{equation}\label{X2}
			\|\tilde{\omega}\|_{X_2}^2 \lesssim \|\tilde{\omega}_0\|_{\mathbb{X}_{0}}+ \|\tilde{\omega}\|_{\mathbb{X} }^4+ \|\tilde{\omega}\|_{\mathbb{X} }^6.
		\end{equation}
		
		\item \textbf{Estimation of $\|\tilde{\omega}\|_{X_3}$.} Applying the higher-order energy estimate with test function $r_\xi e^{\frac{3}{2}\xi}$ to equation \eqref{r}, one derives
		\begin{align*}
			&e^{\frac{3}{2} \xi} \int_0^{\infty} A r_\xi^2 d \psi+\frac{1}{2} \frac{\partial}{\partial \xi}\left(e^{\frac{3}{2} \xi} \int_0^{\infty} \rho r_\psi^2 d \psi\right)-\frac{3}{4} e^{\frac{3}{2} \xi} \int_0^{\infty} \rho r_\psi^2 d \psi\\
			- & e^{\frac{3}{2}\xi} \int_0^{\infty} \frac{\bar{\omega}_{\psi \psi}}{2 \sqrt{\bar{\omega}}} A r r_\xi d \psi =e^{\frac{3}{2}\xi} \int_0^{\infty} A F_\xi r_\xi d \psi.
		\end{align*}
		This generates the following inequality:
		\begin{align*}
			& e^{\frac{3}{2} \xi} \int_0^{\infty} A r_\xi^2 d \psi+\frac{1}{2} \frac{\partial}{\partial \xi}\left(e^{\frac{3}{2} \xi} \int_0^{\infty} \rho r_\psi^2 d \psi\right)\\
			\lesssim & e^{\frac{3}{2} \xi} \int_0^{\infty} \rho r_\psi^2 d \psi +e^{\frac{3}{2}\xi} \int_0^{\infty} A r^2 d \psi + e^{\frac{3}{2}\xi} \int_0^{\infty} A F_\xi^2 d \psi.
		\end{align*}
		Integrating with respect to $\xi$ yields
		\begin{equation}\label{X31}
			\begin{aligned}
				&\sup _{\xi \geq d}\left\|\sqrt{\rho} \tilde{\omega}_{\xi\psi} e^{\frac{3}{4}\xi}\right\|_{L_\psi^2}^2+\left\|\sqrt{A} \tilde{\omega}_{\xi\xi} e^{\frac{3}{4}\xi}\right\|_{L_{\xi,\psi}^2}^2\\
				\lesssim & \left\|\sqrt{\rho} \tilde{\omega}_{0\xi\psi}\right\|_{L_\psi^2}^2 + \|\tilde{\omega}\|_{X_1}^2 + \|\tilde{\omega}\|_{X_2}^2+\left\|\sqrt{A} F_\xi e^{\frac{3}{4} \xi}\right\|_{L_{\xi,\psi}^2}^2.
			\end{aligned}	
		\end{equation}

        From the derivative equation \eqref{r1}, one has
		\begin{equation}\label{X32}
			\begin{aligned}
				&\left\|\frac{\sqrt{A} \sqrt{\bar{\omega}}}{1+\psi} \tilde{\omega}_{\xi\psi\psi} e^{\frac{3}{4}\xi}\right\|_{L_{\xi,\psi}^2}^2\\
				\lesssim & \left\|\sqrt{A} \tilde{\omega}_{\xi\xi} e^{\frac{3}{4}\xi}\right\|_{L_{\xi,\psi}^2}^2 + \left\|\sqrt{\rho} \tilde{\omega}_{\xi\psi} e^{\frac{3}{4}\xi}\right\|_{L_{\xi,\psi}^2}^2 + \left\|\sqrt{A} \tilde{\omega}_\xi e^{\frac{3}{4}\xi}\right\|_{L_{\xi,\psi}^2}^2 + \left\|\sqrt{A} F_\xi e^{\frac{3}{4}\xi}\right\|_{L_{\xi,\psi}^2}^2\\
				\lesssim & \left\|\sqrt{A} \tilde{\omega}_{\xi\xi} e^{\frac{3}{4}\xi}\right\|_{L_{\xi,\psi}^2}^2 + \|\tilde{\omega}\|_{X_1}^2 + \|\tilde{\omega}\|_{X_2}^2+\left\|\sqrt{A} F_\xi e^{\frac{3}{4} \xi}\right\|_{L_{\xi,\psi}^2}^2. 
			\end{aligned}
		\end{equation}
		Synthesizing estimates \eqref{X31}, \eqref{X32} with \eqref{X1}, \eqref{X2} and \eqref{AFX} completes the highest-order estimate:
		\begin{equation}\label{X3}
			\|\tilde{\omega}\|_{X_2}^2 \lesssim \|\tilde{\omega}_0\|_{\mathbb{X}_{0}} + \|\tilde{\omega}\|_{\mathbb{X} }^4+ \|\tilde{\omega}\|_{\mathbb{X} }^6.
		\end{equation}

		\item \textbf{Contraction mapping.} Define the solution map $ \mathcal{T} :H^3(G)\rightarrow H^3(G) $ by $\mathcal{T}(\tilde{\omega})= \Pi$ through the parabolic system:
		\begin{align*}
			&\Pi_\xi-\frac{1}{2}\psi \Pi_\psi -\sqrt{\bar{\omega}}\Pi_{\psi\psi}-\frac{\bar{\omega}_{\psi\psi}}{2\sqrt{\bar{\omega}}}\Pi =F(\tilde{\omega}),\\
			&\Pi(\xi,0)=\Pi(\xi,\infty)=0, \quad \Pi_0(\psi)=\tilde{\omega}_0(\psi).
		\end{align*}
		
		Let
		\begin{equation*}
			B=\left\{\tilde{\omega}\in H^3(G) \left| \| \tilde{\omega} \|_{\mathbb{X}}\leq C_0 \kappa \right.\right\}.
		\end{equation*}
		We will show that $\mathcal{T}$ is a contraction mapping in $B$ if 
		\begin{equation*}
			\|\tilde{\omega}_0\|_{\mathbb{X}_{0}} \lesssim \kappa \ll 1.
		\end{equation*}
		It follows from (\ref{X0}), (\ref{X1}), (\ref{X2}), (\ref{X3}) that
		\begin{align*}
			\| \Pi \|_{\mathbb{X}}^2 & \lesssim \|\tilde{\omega}_0\|_{\mathbb{X}_{0}}^2 + \| \tilde{\omega} \|_{\mathbb{X}}^4\\
			&\leq C_1 (1+C_0^4 \kappa^2)\kappa^2.
		\end{align*}
		Choosing the contraction constant $C_0^2=2C_1$ guarantees $\mathcal{T}(B) \subset B$ for sufficiently small $\kappa$. 
		
		For any $\tilde{\omega}_1, \tilde{\omega}_2 \in B$, the difference estimate follows:
		\begin{align*}
			\left\| \mathcal{T}(\tilde{\omega}_1)-\mathcal{T}(\tilde{\omega}_2)\right\|_{\mathbb{X}}^2 & \lesssim  \| \tilde{\omega}_1 - \tilde{\omega}_2 \|_{\mathbb{X}}^2\left(\| \tilde{\omega}_1 \|_{\mathbb{X}}^2+\| \tilde{\omega}_2 \|_{\mathbb{X}}^2\right)\\
			&\lesssim C_0^2 \kappa^2  \| \tilde{\omega}_1 - \tilde{\omega}_2 \|_{\mathbb{X}}^2.
		\end{align*}
		This establishes the strict contraction property when $\kappa$ is small enough. By Banach's fixed point theorem, the error equation \eqref{error} admits a unique solution in $B$.
	\end{enumerate}
\end{proof}

\subsection{Sharp Decay Estimates of $\tilde{\omega}$}
\ 
\newline
\indent 
In this subsection, we derive the sharp decay rates of $\tilde{\omega}$ from Theorem \ref{sta}:
\begin{proposition}\label{sharp1}
	Under the assumptions in Theorem \ref{sta}, the remainder $\tilde{\omega}$ satisfies the following sharp decay estimates:
	\begin{equation}
		\left\|\sqrt{A} \tilde{\omega} \right\|_{L_\psi^2} \lesssim e^{-\xi}, \quad \left\|\sqrt{\rho} \tilde{\omega}_\psi \right\|_{L_\psi^2} \lesssim e^{-\xi}.
	\end{equation}
\end{proposition}
\begin{proof}
    From Theorem \ref{sta}, we inherit the fundamental estimate:
	\begin{equation*}
		\left\|\sqrt{A} \tilde{\omega} \right\|_{L_\psi^2} \lesssim e^{-\xi}.
	\end{equation*}
	
	Testing equation (\ref{wei}) with $\tilde{\omega} e^{\frac{5}{2}\xi}$ generates the integral identity:
	\begin{align*}
		&\frac{1}{2} \frac{\partial}{\partial \xi}\left(e^{\frac{5}{2}\xi} \int_0^{\infty} A \tilde{\omega}^2 d \psi\right)-\frac{3}{4}e^{\frac{5}{2}\xi} \int_0^{\infty} A \tilde{\omega}^2 d \psi\\
		+&e^{\frac{5}{2} \xi} \int_0^{\infty} \rho \tilde{\omega}_\psi^2 d \psi-e^{\frac{5}{2}\xi}\int_0^{\infty} \frac{\bar{\omega}_{\psi\psi}}{2\sqrt{\bar{\omega}}} A \tilde{\omega}^2 d \psi= e^{\frac{5}{2} \xi} \int_0^{\infty} A F \tilde{w} d \psi,
	\end{align*}
	which leads to the simplified inequality:
	\begin{align*}
		&\frac{1}{2} \frac{\partial}{\partial \xi}\left(e^{\frac{5}{2}\xi} \int_0^{\infty} A \tilde{\omega}^2 d \psi\right)+e^{\frac{5}{2} \xi} \int_0^{\infty} \rho \tilde{\omega}_\psi^2 d \psi \\
		\lesssim & e^{\frac{5}{2}\xi} \int_0^{\infty} A \tilde{\omega}^2 d \psi+e^{\frac{5}{2} \xi} \int_0^{\infty} A F^2 d \psi.
	\end{align*}	
	Integrating with respect to $\xi$ yields:
	\begin{align*}
		&e^{\frac{5}{2}\xi} \int_0^{\infty} A \tilde{\omega}^2 d \psi+\left\|\sqrt{\rho} \tilde{\omega}_\psi e^{\frac{5}{4} \xi}\right\|_{L_{\xi,\psi}^2}^2\\
		\lesssim & \left\|\sqrt{A} \tilde{\omega}_0\right\|_{L_\psi^2}^2+e^{\frac{1}{2}\xi}\|\tilde{\omega}\|_{X_0}^2+\|\tilde{\omega}\|_{\mathbb{X} }^4\\
		\lesssim & e^{\frac{1}{2}\xi}.
	\end{align*}
	
	Subsequently, testing equation (\ref{wei}) with $\tilde{\omega}_\xi e^{\frac{5}{2}\xi}$ and integrating by parts produces:
	\begin{align*}
		&e^{\frac{5}{2} \xi} \int_0^{\infty} A \tilde{\omega}_\xi^2 d \psi+\frac{1}{2} \frac{\partial}{\partial \xi}\left(e^{\frac{5}{2} \xi} \int_0^{\infty} \rho \tilde{\omega}_\psi^2 d \psi\right)-\frac{5}{4} e^{\frac{5}{2} \xi} \int_0^{\infty} \rho \tilde{\omega}_\psi^2 d \psi\\
		-&e^{\frac{5}{2}\xi} \int_0^{\infty} \frac{\bar{\omega}_{\psi \psi}}{2 \sqrt{\bar{\omega}}} A \tilde{\omega} \tilde{\omega}_\xi d \psi =e^{\frac{5}{2}} \int_0^{\infty} A F \tilde{\omega}_\xi d \psi.
	\end{align*}
	This generates the following inequality:
	\begin{align*}
		&e^{\frac{5}{2} \xi} \int_0^{\infty} A \tilde{\omega}_\xi^2 d \psi+\frac{1}{2} \frac{\partial}{\partial \xi}\left(e^{\frac{5}{2} \xi} \int_0^{\infty} \rho \tilde{\omega}_\psi^2 d \psi\right)\\
		\lesssim & e^{\frac{5}{2} \xi} \int_0^{\infty} \rho \tilde{\omega}_\psi^2 d \psi + e^{\frac{5}{2} \xi} \int_0^{\infty} A \tilde{\omega}^2 d \psi+e^{\frac{5}{2}\xi} \int_0^{\infty} A F^2 d \psi.
	\end{align*}
	Integrating with respect to $\xi$ yields:
	\begin{align*}
		&e^{\frac{5}{2} \xi} \int_0^{\infty} \rho \tilde{\omega}_\psi^2 d \psi + \left\|\sqrt{A} \tilde{\omega}_\xi e^{\frac{5}{4} \xi}\right\|_{L_{\xi,\psi}^2}^2\\
		\lesssim & \left\|\sqrt{\rho} \tilde{\omega}_{0\psi}\right\|_{L_\psi^2}^2 + \left\|\sqrt{\rho} \tilde{\omega}_\psi e^{\frac{5}{4} \xi}\right\|_{L_{\xi,\psi}^2}^2 +e^{\frac{1}{2}\xi} +\|\tilde{\omega}\|_{\mathbb{X} }^4\\
		\lesssim & e^{\frac{1}{2}\xi},
	\end{align*}
	then one has
	\begin{equation*}
		\left\|\sqrt{\rho} \tilde{\omega}_\psi \right\|_{L_\psi^2} \lesssim e^{-\xi}.
	\end{equation*}
\end{proof}

\begin{proposition}\label{sharp2}
	Under the assumptions in Theorem \ref{sta}, the following sharp decay estimates hold:
	\begin{equation}
		\left\|\sqrt{A} \tilde{\omega}_\xi\right\|_{L_\psi^2} \lesssim e^{-\xi}, \quad \left\|\frac{\sqrt{A} \sqrt{\bar{\omega}}}{1+\psi} \tilde{\omega}_{\psi\psi} \right\|_{L_\psi^2} \lesssim e^{-\xi}.
	\end{equation}
\end{proposition}
\begin{proof}
	Testing equation (\ref{r}) with $r e^{\frac{5}{2}\xi}$ and integrating by parts yields:
	\begin{align*}
		&\frac{1}{2} \frac{\partial}{\partial \xi}\left(e^{\frac{5}{2} \xi} \int_0^{\infty} A r^2 d \psi\right)-\frac{5}{4} e^{\frac{5}{2} \xi} \int_0^{\infty} A r^2 d \psi + e^{\frac{5}{2} \xi} \int_0^{\infty} \rho r_\psi^2 d \psi\\
		- & e^{\frac{5}{2}\xi} \int_0^{\infty} \frac{\bar{\omega}_{\psi \psi}}{2 \sqrt{\bar{\omega}}} A r^2 d \psi =e^{\frac{5}{2}\xi} \int_0^{\infty} A F_\xi r d \psi,
	\end{align*}
	which leads to the following inequality:
	\begin{align*}
		& \frac{\partial}{\partial \xi}\left(e^{\frac{5}{2} \xi} \int_0^{\infty} A r^2 d \psi\right)+ e^{\frac{5}{2} \xi} \int_0^{\infty} \rho r_\psi^2 d \psi\\
		\lesssim & e^{\frac{5}{2} \xi} \int_0^{\infty} A r^2 d \psi +e^{\frac{5}{2}\xi} \int_0^{\infty} A F_\xi^2 d \psi.
	\end{align*}
	Integrating with respect to $\xi$ establishes:
	\begin{align*}
		&e^{\frac{5}{2} \xi} \int_0^{\infty} A r^2 d \psi+\left\|\sqrt{\rho} \tilde{\omega}_{\xi\psi} e^{\frac{5}{4}\xi}\right\|_{L_{\xi,\psi}^2}^2\\
		\lesssim & \left\|\sqrt{A} \tilde{\omega}_{0\xi}\right\|_{L_\psi^2}^2 + \left\|\sqrt{A} \tilde{\omega}_\xi e^{\frac{5}{4} \xi}\right\|_{L_{\xi,\psi}^2}^2 +\|\tilde{\omega}\|_{\mathbb{X} }^4\\
		\lesssim & e^{\frac{1}{2}\xi}.
	\end{align*}	
	This implies the first decay estimate:
	\begin{equation*}
		\left\|\sqrt{A} \tilde{\omega}_\xi\right\|_{L_\psi^2} \lesssim e^{-\xi}.
	\end{equation*}

	Examining equation (\ref{error}) reveals the relationship:
	\begin{equation*}
		\left\|\frac{\sqrt{A} \sqrt{\bar{\omega}}}{1+\psi} \tilde{\omega}_{\psi\psi} \right\|_{L_\psi^2} \lesssim \left\|\sqrt{A} \tilde{\omega}_\xi\right\|_{L_\psi^2} +\left\|\sqrt{\rho} \tilde{\omega}_\psi \right\|_{L_\psi^2}+\left\|\sqrt{A} \tilde{\omega} \right\|_{L_\psi^2}+\left\|\sqrt{A} F \right\|_{L_\psi^2}.
	\end{equation*}
	
	Previous analysis in Lemma \ref{exi} demonstrates:
	\begin{equation*}
		\left\|\sqrt{A} F \right\|_{L_\psi^2} \lesssim e^{-3\xi} \|\tilde{\omega}\|_{\mathbb{X} }^4.
	\end{equation*}
	Consequently, one establishes the second decay estimate:
	\begin{equation*}
		\left\|\frac{\sqrt{A} \sqrt{\bar{\omega}}}{1+\psi} \tilde{\omega}_{\psi\psi} \right\|_{L_\psi^2} \lesssim e^{-\xi}.
	\end{equation*}
\end{proof}

\begin{proposition}\label{sharp3}
	Under the assumptions in Theorem \ref{sta}, the mixed derivative satisfies:
	\begin{equation}
		\left\|\sqrt{\rho} \tilde{\omega}_{\xi\psi} \right\|_{L_\psi^2} \lesssim e^{-\xi}.
	\end{equation}
\end{proposition}
\begin{proof}
    Applying the test function $r_\xi e^{\frac{5}{2}\xi}$ to (\ref{r}) and integrating by parts produces:
	\begin{align*}
		&e^{\frac{5}{2} \xi} \int_0^{\infty} A r_\xi^2 d \psi+\frac{1}{2} \frac{\partial}{\partial \xi}\left(e^{\frac{5}{2} \xi} \int_0^{\infty} \rho r_\psi^2 d \psi\right)-\frac{5}{4} e^{\frac{5}{2} \xi} \int_0^{\infty} \rho r_\psi^2 d \psi\\
		- & e^{\frac{5}{2}\xi} \int_0^{\infty} \frac{\bar{\omega}_{\psi \psi}}{2 \sqrt{\bar{\omega}}} A r r_\xi d \psi =e^{\frac{5}{2}\xi} \int_0^{\infty} A F_\xi r_\xi d \psi,
	\end{align*}
	which leads to the simplified inequality:
	\begin{align*}
		& e^{\frac{5}{2} \xi} \int_0^{\infty} A r_\xi^2 d \psi+\frac{1}{2} \frac{\partial}{\partial \xi}\left(e^{\frac{5}{2} \xi} \int_0^{\infty} \rho r_\psi^2 d \psi\right)\\
		\lesssim & e^{\frac{5}{2} \xi} \int_0^{\infty} \rho r_\psi^2 d \psi +e^{\frac{5}{2}\xi} \int_0^{\infty} A r^2 d \psi + e^{\frac{5}{2}\xi} \int_0^{\infty} A F_\xi^2 d \psi.
	\end{align*}
	Integrating with respect to $\xi$ gives:
	\begin{align*}
		&e^{\frac{5}{2} \xi} \int_0^{\infty} \rho r_\psi^2 d \psi+\left\|\sqrt{A} r_{\xi} e^{\frac{5}{4}\xi}\right\|_{L_{\xi,\psi}^2}^2\\
		\lesssim & \left\|\sqrt{\rho} \tilde{\omega}_{0\xi\psi}\right\|_{L_\psi^2}^2 + \left\|\sqrt{\rho} r_{\psi} e^{\frac{5}{4}\xi}\right\|_{L_{\xi,\psi}^2}^2 + e^{\frac{1}{2}\xi}+\|\tilde{\omega}\|_{\mathbb{X} }^4\\
		\lesssim & e^{\frac{1}{2}\xi}.
	\end{align*}	
	This final estimate proves the required decay property:
	\begin{equation*}
		\left\|\sqrt{\rho} \tilde{\omega}_{\xi\psi} \right\|_{L_\psi^2} \lesssim e^{-\xi}.
	\end{equation*}
\end{proof}

\subsection{Sharp $L^\infty$ Decay Estimates of $u$}
\ 
\newline
\indent 
In this section, we demonstrate our main result, Theorem \ref{energy1}.
\begin{proof}[Proof of Theorem \ref{energy1}]
    Under the hypotheses of Theorem \ref{sta} and through applications of Proposition \ref{sharp1}-\ref{sharp3}, one establishes the following low-regularity estimates:
	\begin{equation}
		\left\|\partial_y^j (u-\bar{u}) (x,y)\right\|_{L^{\infty}_y}\lesssim (x+1)^{-1-\frac{j}{2}}, \quad j=0,1,2.
	\end{equation}
	
	Our previous decay results in Propositions \ref{sharp1}-\ref{sharp3} yield essential $L^\infty$ bounds:

	\begin{align*}
		\left\| \tilde{\omega} \right\|_{L^{\infty}_\psi} &\lesssim \left\|\sqrt{\rho} \tilde{\omega}_\psi \right\|_{L_\psi^2} \lesssim e^{-\xi},\\
		\left\| \tilde{\omega}_\psi \right\|_{L^{\infty}_\psi} &\lesssim \left\|\frac{\sqrt{A} \sqrt{\bar{\omega}}}{1+\psi} \tilde{\omega}_{\psi\psi} \right\|_{L_\psi^2} \lesssim e^{-\xi},\\
		\left\| \tilde{\omega}_\xi \right\|_{L^{\infty}_\psi} &\lesssim \left\|\sqrt{\rho} \tilde{\omega}_{\xi\psi} \right\|_{L_\psi^2} \lesssim e^{-\xi}.\\
	\end{align*}
	
	Through the modified von Mises transformation:
	\begin{equation*}
		u=\sqrt{\omega},\quad x+1=e^{\xi}, \quad y=\int_{0}^{\psi}\frac{e^{\frac{1}{2}\xi}}{\sqrt{\omega(\xi,s)}}d s ,
	\end{equation*}
	standard differential calculus yields the velocity components:
	\begin{align*}
		&u_y=\frac{1}{2}e^{-\frac{1}{2}\xi}\omega_\psi, \quad u_{y y}=\frac{1}{2}e^{-\xi}\sqrt{\omega}\omega_{\psi\psi},\\
		&u_x=e^{-\xi} \frac{\omega_{\xi}}{2\sqrt{\omega}} - e^{-\xi} \frac{\omega_{\psi}}{4} \int_{0}^{\psi}\left(\frac{1}{\omega^\frac{1}{2}}-\frac{\omega_\xi}{\omega^{\frac{3}{2}}}\right)d s,\\
		&v=-e^{-\frac{1}{2}\xi}\frac{\psi}{2} + e^{-\frac{1}{2}\xi}\frac{\sqrt{\omega}}{2}\int_{0}^{\psi}\left(\frac{1}{\omega^\frac{1}{2}}-\frac{\omega_\xi}{\omega^{\frac{3}{2}}}\right)d s.
	\end{align*}

    Firstly, one can derive
	\begin{equation*}
		\left\| u(x,y)-\bar{u}(x,y)\right\|_{L^{\infty}_y}=\left\| \sqrt{\omega}-\sqrt{\bar{\omega}}\right\|_{L^{\infty}_\psi}=\left\| \frac{\tilde{\omega}}{\sqrt{\omega}+\sqrt{\bar{\omega}}}\right\|_{L^{\infty}_\psi}.
	\end{equation*}
	If $\psi\geq 1$, one has
	\begin{equation*}
		\left\| u(x,y)-\bar{u}(x,y)\right\|_{L^{\infty}_y} \lesssim \left\| \tilde{\omega} \right\|_{L^{\infty}_\psi} \lesssim e^{-\xi}.
	\end{equation*}
	If $\psi \leq 1$, one has
	\begin{equation*}
		\left\| u(x,y)-\bar{u}(x,y)\right\|_{L^{\infty}_y} \lesssim \left\| \tilde{\omega}_\psi \right\|_{L^{\infty}_\psi} \lesssim e^{-\xi}.
	\end{equation*}

	One also can derive
	\begin{equation*}
		\left\| u_y(x,y)-\bar{u}_y(x,y)\right\|_{L^{\infty}_y} = \frac{1}{2}e^{-\frac{1}{2}\xi} \left\| \tilde{\omega}_\psi \right\|_{L^{\infty}_\psi} \lesssim e^{-\frac{3}{2}\xi}.
	\end{equation*}

	Next, one has
	\begin{align*}
		\left\| u_{yy}(x,y)-\bar{u}_{yy}(x,y)\right\|_{L^{\infty}_y}&=\frac{1}{2}e^{-\xi}\left\| \sqrt{\omega}\omega_{\psi\psi}-\sqrt{\bar{\omega}}\bar{\omega}_{\psi\psi}\right\|_{L^{\infty}_\psi}\\
		&\lesssim e^{-\xi} \left\|\tilde{\omega}_{\xi}\right\|_{L^{\infty}_\psi} + e^{-\xi} \left\| \psi \tilde{\omega}_\psi \right\|_{L^{\infty}_\psi}\\
		&\lesssim e^{-2\xi}.
	\end{align*}

	Furthermore, we present the proof of the high regularity results. Based on the modified von Mises transformation, the assumptions of initial value $u_0(y)$ in Theorem \ref{energy1} can ensure that the following properties of $\omega_0(\psi)$ hold for $i=0,1,2,\cdots, K_0$:
	\begin{equation}
		\left\| \sqrt{A} \partial_\xi^i \tilde{\omega}_0\right\|_{L^2_\psi}^2 + \left\| \sqrt{\rho} \partial_\xi^i \partial_\psi \tilde{\omega}_0\right\|_{L^2_\psi}^2 \lesssim \kappa \ll 1.
	\end{equation}

	Define the following weighted energy norms for $i=1,2,\cdots,K_0$:
	\begin{align*}
		\|\tilde{\omega}\|_{X_0}^2&=\sup_{\xi \geq d}\left\|\sqrt{A} \tilde{\omega} e^\xi\right\|_{L_\psi^2}^2+\left\|\sqrt{\rho} \tilde{\omega}_\psi e^{\frac{3}{4} \xi}\right\|_{L_{\xi,\psi}^2}^2;\\
		\|\tilde{\omega}\|_{X_{2i-1}}^2&=\sup _{\xi \geq d} \left\|\sqrt{\rho} \partial_\xi^{i-1} \partial_\psi \tilde{\omega} e^{\frac{3}{4}\xi}\right\|_{L_\psi^2}^2+\left\|\sqrt{A} \partial_\xi^{i}\tilde{\omega} e^{\frac{3}{4}\xi}\right\|_{L_{\xi,\psi}^2}^2;\\
		\|\tilde{\omega}\|_{X_{2i}}^2&=\sup _{\xi \geq d}\left\|\sqrt{A} \partial_\xi^{i} \tilde{\omega} e^{\frac{3}{4}\xi}\right\|_{L_\psi^2}^2+\left\|\sqrt{\rho} \partial_\xi^{i} \partial_\psi \tilde{\omega} e^{\frac{3}{4}\xi}\right\|_{L_{\xi,\psi}^2}^2+\sup _{\xi \geq d}\left\|\frac{\sqrt{A} \sqrt{\bar{\omega}}}{1+\psi} \partial_\xi^{i-1} \partial_\psi^2 \tilde{\omega} e^{\frac{3}{4}\xi}\right\|_{L_\psi^2}^2;\\
		\|\tilde{\omega}\|_{X_{2K_0+1}}^2&=\sup _{\xi \geq d}\left\|\sqrt{\rho} \partial_\xi^{K_0} \partial_\psi\tilde{\omega} e^{\frac{3}{4} \xi}\right\|_{L_\psi^2}^2+\left\|\sqrt{A} \partial_\xi^{K_0+1} \tilde{\omega} e^{\frac{3}{4} \xi}\right\|_{L_{\xi,\psi}^2}^2+\left\|\frac{\sqrt{A} \sqrt{\omega}}{1+\psi} \partial_\xi^{K_0} \partial_\psi^2 \tilde{\omega} e^{\frac{3}{4}\xi}\right\|_{L_{\xi, \psi}^2}^2.
	\end{align*}	
	Construct the solution space:
	\begin{equation}
		\|\tilde{\omega}\|_{\mathbb{\tilde{X}} }=\left\{\tilde{\omega}\left|\sum_{i=0}^{2K_0+1}\|\tilde{\omega}\|_{X_i}<\infty \right.\right\}.
	\end{equation}

	Similar to the proof of the Theorem \ref{sta}, one can obtain that the existence of $\tilde{\omega}$ in $\|\tilde{\omega}\|_{\mathbb{\tilde{X}} }$. Then, similar to the proofs of Proposition \ref{sharp1}-\ref{sharp3}, one has
	\begin{align*}
		&\left\|\sqrt{A} \partial_\xi^{i} \tilde{\omega} \right\|_{L_\psi^2} \lesssim e^{-\xi},\quad i=0,1,2,\cdots,K_0,\\
		&\left\|\sqrt{\rho} \partial_\xi^{i} \partial_\psi \tilde{\omega}\right\|_{L_\psi^2}\lesssim e^{-\xi},\quad i=0,1,2,\cdots,K_0,\\
		&\left\|\frac{\sqrt{A} \sqrt{\bar{\omega}}}{1+\psi} \partial_\xi^{i} \partial_\psi^2 \tilde{\omega}\right\|_{L_\psi^2}\lesssim e^{-\xi},\quad i=0,1,2,\cdots,K_0-1,\\
		&\left\|\frac{\sqrt{A}}{1+\psi} \bar{\omega}^{\frac{2j-3}{2}} \partial_\xi^{i} \partial_\psi^j \tilde{\omega}\right\|_{L_\psi^2}\lesssim e^{-\xi},\quad j\geq 3, \quad 2i+j\leq 2K_0.
	\end{align*}
	
	Leveraging our established $L^2$ estimates and equation \eqref{error}, one rigorously derives these $L^\infty$ decay bounds:
	\begin{align*}
		&\left\|\partial_\xi^{i} \tilde{\omega} \right\|_{L_\psi^\infty} \lesssim e^{-\xi}, \quad i=0,1,2,\cdots,K_0,\\
		&\left\|\partial_\xi^{i} \partial_\psi \tilde{\omega}\right\|_{L_\psi^\infty} \lesssim e^{-\xi},\quad i=0,1,2,\cdots,K_0-1,\\
		&\left\|\bar{\omega}^{\frac{2j-3}{2}} \partial_\xi^{i} \partial_\psi^j \tilde{\omega}\right\|_{L_\psi^\infty} \lesssim e^{-\xi}, \quad j\geq 2, \quad 2i+j\leq 2K_0.
	\end{align*}
	
	Finally, through inverse von Mises transformation, one establishes the main regularity result:
	\begin{equation}
		\left\|\partial_x^i \partial_y^j (u-\bar{u}) (x,y)\right\|_{L_y^\infty} \lesssim (x+1)^{-1-i-\frac{j}{2}},\quad 2i+j\leq 2K_0-1.
	\end{equation}
\end{proof}

\section{Proof of Theorem \ref{main1}}

Let $\tilde{\omega}=\omega-\bar{\omega}$, where
\begin{equation}
	\bar{\omega}\left(\int_{0}^{y}\frac{u(x,y')}{\sqrt{x+d}}dy'\right)=\left[f'\left(\frac{y}{\sqrt{2(x+d)}}\right)\right]^2.
\end{equation}
One can obtain
\begin{equation}\label{L1}
	\mathcal{L} _1 (\tilde{\omega})=\tilde{\omega}_\xi-\frac{1}{2}\psi\tilde{\omega}_\psi-\sqrt{\omega}\tilde{\omega}_{\psi\psi}-\frac{\bar{\omega}_{\psi\psi}}{\sqrt{\bar{\omega}}+\sqrt{\omega}}\tilde{\omega}=0,
\end{equation}
with boundary conditions
\begin{equation*}
	\tilde{\omega}(\xi,0)=\tilde{\omega}(\xi,\infty)=0, \quad \tilde{\omega}(0,\psi)=\omega_0(\psi)-\bar{\omega}(\psi).
\end{equation*}

We first prove a low-decay estimate of $\omega$.
\begin{lemma}\label{w1}
	Under the hypotheses of Theorem \ref{main1}, there exists small positive constants $\alpha$ and $\mu$ such that the following inequality holds:
	\begin{equation}
		|\omega(\xi,\psi)-\bar{\omega}(\psi)|\lesssim_\varepsilon e^{-\alpha\xi}e^{-\mu \psi^2}\bar{\omega}.
	\end{equation}
\end{lemma}
\begin{proof}
    Define the comparison function 
    \begin{equation}
         \varphi_1:= M_1 e^{-\alpha\xi}e^{-\mu \psi^2}\bar{\omega}
    \end{equation}
    and consider the barrier function $s_1 := \varphi_1 \pm \tilde{\omega}$. The boundary conditions satisfy:
    \begin{equation*}
          s_1(\xi,0) = s_1(\xi,\infty) = 0.
    \end{equation*}
    Through the modified von Mises transformation:
    \begin{equation*}
        \psi(x,y) = \int_0^y \frac{u(x,y')}{\sqrt{x+d}} dy',
    \end{equation*}
    one establishes the asymptotic correspondence $\psi \sim y/\sqrt{x+d}$ as $y\to\infty$. The initial positivity follows from assumptions \eqref{suppose1} by selecting sufficiently large $M_1(\varepsilon)$ and $d(\varepsilon)$. 
	
	Computing the operator action on $\varphi_1$:
	\begin{align*}
		\mathcal{L} _1(\varphi_1)=&-\alpha M_1 e^{-\alpha\xi}e^{-\mu \psi^2}\bar{\omega}+M_1 e^{-\alpha\xi}e^{-\mu \psi^2}\left(-\frac{1}{2}\psi\bar{\omega}_\psi-\sqrt{\omega}\bar{\omega}_{\psi\psi}-\frac{\bar{\omega}_{\psi\psi}}{\sqrt{\bar{\omega}}+\sqrt{\omega}}\bar{\omega}\right)\\
		&+M_1 e^{-\alpha\xi}e^{-\mu \psi^2}\bar{\omega}\left(\mu\psi^2+2\mu\sqrt{\omega}-4\sqrt{\omega}\mu^2\psi^2\right)+M_1 e^{-\alpha\xi}e^{-\mu \psi^2}\left(4\mu\psi\sqrt{\omega}\right) \bar{\omega}_\psi\\
		\geq &M_1 e^{-\alpha\xi}e^{-\mu \psi^2}\left(-\alpha\bar{\omega}-\frac{\bar{\omega}_{\psi\psi}}{\sqrt{\bar{\omega}}+\sqrt{\omega}}\omega+2\mu\sqrt{\omega}\bar{\omega}\right).
	\end{align*}
	
	Through Lemma \ref{lemma2}, one partitions the domain at $\psi_0>0$ where:
    \begin{equation*}
        \begin{cases}
            -\tfrac{\bar{\omega}_{\psi\psi}}{\sqrt{\bar{\omega}}+\sqrt{\omega}}\omega \geq \alpha\bar{\omega}, & 0\leq\psi\leq\psi_0 \\
            2\mu\sqrt{\omega}\bar{\omega} \geq \alpha\bar{\omega}, & \psi\geq\psi_0
        \end{cases}
    \end{equation*}
    yielding $\mathcal{L}_1(\varphi_1) \geq 0$. Consequently:
	\begin{equation*}
		\mathcal{L} _1(s_1)=\mathcal{L} _1(\varphi_1)\pm \mathcal{L} _1(\tilde{\omega})\geq 0.
	\end{equation*}

    Application of the maximum principle to the operator $\mathcal{L}_1$ yields $s_1\geq0$ throughout the domain, establishing:
	\begin{equation*}
		|\omega(\xi,\psi)-\bar{\omega}(\psi)|\leq M_1(\varepsilon) e^{-\alpha\xi}e^{-\mu \psi^2}\bar{\omega}.
	\end{equation*}
\end{proof}

By utilizing the decay estimate in Lemma \ref{w1}, we can obtain the following sharp decay estimate of $\omega(\xi,\psi)$.
\begin{proposition}\label{wsd}
	Under the hypotheses of Theorem \ref{main1}, there exists a positive constant $\mu$ such that the following inequality holds :
	\begin{equation}\label{wsharp}
		|\omega(\xi,\psi)-\bar{\omega}(\psi)|\lesssim_\varepsilon e^{-\xi} \psi e^{-\mu \psi^2}.
	\end{equation}
\end{proposition}
\begin{proof}
    Define the enhanced comparison function
    \begin{equation}
        \varphi_2(\xi,\psi) := M_2 e^B e^{- \xi} \psi \bar{\omega}_{\psi} e^{-B \exp(-\delta \xi)}
    \end{equation}
    with parameters $\delta$ and $B$ to be determined. Consider the barrier function $s_2 := \varphi_2 \pm \tilde{\omega}$ satisfying:
    \begin{equation*}
        s_2(\xi,0) = s_2(\xi,\infty) = 0, \quad s_2(\ln d,\psi) \geq 0,
    \end{equation*}
    for sufficiently large $M_2(\varepsilon)$ and $d(\varepsilon)$. 

	Computing the operator action on $\varphi_2$:
	\begin{align*}
		\mathcal{L} _1(\varphi_2)=&-\varphi_2+\delta B e^{-\delta \xi}\varphi_2+ M_2 e^B e^{- \xi} e^{-B \exp(-\delta \xi)} \mathcal{L} _1 (\psi \bar{\omega}_{\psi})\\
		=& \varphi_2 \left(\frac{\bar{\omega}_{\psi\psi}\tilde{\omega}}{2\sqrt{\bar{\omega}}(\sqrt{\omega}+\sqrt{\bar{\omega}})^2}-\frac{\tilde{\omega}}{\sqrt{\omega}+\sqrt{\bar{\omega}}}\frac{\psi \bar{\omega}_{\psi\psi\psi}+2\bar{\omega}_{\psi\psi}}{\psi \bar{\omega}_{\psi}}+\delta B e^{-\delta \xi}\right)\\
		=& \varphi_2 \left(I_1+I_2+\delta B e^{-\delta \xi}\right),
	\end{align*}
	where 
	\begin{align*}
		\mathcal{L} _1 (\psi \bar{\omega}_{\psi}) = & -\frac{1}{2}\psi (\psi\bar{\omega}_{\psi\psi}+\bar{\omega}_{\psi})-\sqrt{\omega}(\psi\bar{\omega}_{\psi\psi\psi}+2\bar{\omega}_{\psi\psi})-\frac{\bar{\omega}_{\psi\psi}}{\sqrt{\bar{\omega}}+\sqrt{\omega}} \psi \bar{\omega}_{\psi}\\
		=& (\sqrt{\bar{\omega}}-\sqrt{\omega}) (\psi\bar{\omega}_{\psi\psi\psi}+2\bar{\omega}_{\psi\psi})+ \psi \bar{\omega}_{\psi}+\left(\frac{1}{2\sqrt{\bar{\omega}}}-\frac{1}{\sqrt{\bar{\omega}}+\sqrt{\omega}}\right) \bar{\omega}_{\psi\psi}\psi \bar{\omega}_{\psi}\\
		=&\psi \bar{\omega}_{\psi}+\frac{\bar{\omega}_{\psi\psi}\tilde{\omega}}{2\sqrt{\bar{\omega}}(\sqrt{\omega}+\sqrt{\bar{\omega}})^2}\psi \bar{\omega}_{\psi}-\frac{\tilde{\omega}}{\sqrt{\omega}+\sqrt{\bar{\omega}}}\left(\psi \bar{\omega}_{\psi\psi\psi}+2\bar{\omega}_{\psi\psi}\right).
	\end{align*} 

	By Lemma \ref{barw}, Lemma \ref{lemma1}, Lemma \ref{lemma2} and Lemma \ref{w1}, one can derive
	\begin{align*}
		&|I_1| \lesssim \left|\frac{\bar{\omega} \bar{\omega}_{\psi\psi}}{2\sqrt{\bar{\omega}}(\sqrt{\omega}+\sqrt{\bar{\omega}})^2} \right| \left|\frac{\tilde{\omega}}{\bar{\omega}} \right| \lesssim e^{-\alpha\xi},\\
		&|I_2| \lesssim \left| \frac{\psi \bar{\omega}_{\psi\psi\psi}+2\bar{\omega}_{\psi\psi}}{\psi \bar{\omega}_{\psi}} \frac{\bar{\omega}}{\sqrt{\omega}+\sqrt{\bar{\omega}}}\right| \left|\frac{\tilde{\omega}}{\bar{\omega}}\right|\lesssim  e^{-\alpha\xi}.
	\end{align*}
    
    Select $\delta \in (0,\alpha]$ and $B>0$ satisfying $\delta B$ sufficiently large to ensure:
	\begin{equation*}
		\mathcal{L} _1(\varphi_2) \geq \varphi_2 \left(I_1+I_2+\delta B e^{-\delta \xi}\right) \geq 0.
	\end{equation*}
	Consequently,
	\begin{equation*}
		\mathcal{L} _1(s_2)=\mathcal{L} _1(\varphi_2)\pm \mathcal{L} _1(\tilde{\omega})\geq 0.
	\end{equation*}

	Applying the maximum principle to $\mathcal{L}_1$ yields $s_2 \geq 0$, thus
	\begin{equation*}
		|\omega(\xi,\psi)-\bar{\omega}(\psi)|\leq M_2(\varepsilon) e^B e^{- \xi} \psi \bar{\omega}_{\psi} e^{-B \exp(-\delta \xi)} \leq C(\varepsilon) e^{-\xi} \psi e^{-\mu \psi^2}
	\end{equation*}
	by taking suitable $C$ and $\mu$.
\end{proof}

Next, we demonstrate our main result, Theorem \ref{main1}.
\begin{proof}[Proof of Theorem \ref{main1}]
	Through the inverse von Mises transformation and Theorem \ref{wsd}, one can obtain the velocity field estimate for suitable positive constants $\mu(\varepsilon)$ and $d(\varepsilon)$:
	\begin{align*}
		\left\|u(x,y)-\bar{u}(x,y)\right\|_{L_y^\infty}& \leq \left\|u(x,y)-\bar{u}^d(x,y)\right\|_{L_y^\infty} +\left\|\bar{u}^d(x,y)-\bar{u}(x,y)\right\|_{L_y^\infty}\\
		& \lesssim \left\|\frac{\omega-\bar{\omega}}{\sqrt{\omega}+\sqrt{\bar{\omega}}}\right\|_{L_\psi^\infty} + (x+1)^{-1}\\
		& \lesssim_\varepsilon e^{-\xi} + (x+1)^{-1} \\
		& \lesssim_\varepsilon (x+1)^{-1}.
	\end{align*}
	where 
	\begin{equation}
		\bar{u}^d (x,y)=f'\left(\frac{y}{\sqrt{2(x+d)}}\right).
	\end{equation}
\end{proof}

\section{Proof of Theorem \ref{main2}}

\subsection{Sharp Decay Estimates of $\omega_\xi$, $\omega_\psi$ and $\omega_{\psi\psi}$}
\ 
\newline
\indent 
Define the derivative quantity $r := \omega_\xi$ satisfying the linear operator equation:
\begin{equation}\label{l2}
	\mathcal{L} _2 (r)= r_\xi-\frac{1}{2}\psi r_\psi-\sqrt{\omega} r_{\psi\psi}-\frac{\omega_{\psi\psi}}{2\sqrt{\omega}} r=0,
\end{equation}
with boundary conditions
\begin{equation*}
	 r (\xi,0)= r(\xi,\infty)=0. 
\end{equation*}

Similar to $\omega$, we first present the results of low-decay estimate.
\begin{lemma}\label{wx1}
	Under the hypotheses of Theorem \ref{main2}, there exists positive constants $\mu$ and $\alpha$ such that:
	\begin{equation}
		\begin{gathered}
			\left|\omega_\xi (\xi,\psi)\right| \lesssim_\varepsilon e^{-\alpha\xi} \psi e^{-\mu \psi^2},\\
			\left|\omega_{\psi}-\bar{\omega}_\psi\right| \lesssim_\varepsilon e^{-\alpha\xi} e^{-\mu \psi^2},\\
			\left|\sqrt{\bar{\omega}}(\omega_{\psi\psi}-\bar{\omega}_{\psi\psi})\right| \lesssim_\varepsilon e^{-\alpha\xi} \psi e^{-\mu \psi^2}.
		\end{gathered}
	\end{equation}
\end{lemma}
\begin{proof}
    Define the comparison function:
    \begin{equation}
        \varphi_3 := M_3 e^{-\alpha\xi} \left[ \left(\psi-\psi^{\frac{5}{4}}\right) \mathcal{X}(\psi) + N e^{-\mu \psi^2}\right]
    \end{equation}
    where $\mathcal{X}(\psi)$ is a smooth cutoff function defined as
    
	\begin{equation*}
		\mathcal{X}(\psi)=\left\{
			\begin{aligned}
				1, &\quad 0\leq \psi \leq \frac{\psi_0}{2},\\
				0, & \quad \psi \geq \psi_0,
			\end{aligned}
		\right.   \quad  \left|\mathcal{X}' (\psi)\right|\lesssim 1, \left|\mathcal{X}'' (\psi)\right|\lesssim 1,
	\end{equation*}
	and $M_3$, $\psi_0$, $N$, $\mu$ and $\alpha$ are positive constants to be determined. 
	Construct the  barrier function $s_3 := \varphi_3 \pm r$ satisfying:
	\begin{equation*}
		s_3(\xi,0)\geq 0,\quad s_3(\xi,\infty)=0,\quad s_3(\ln d,\psi) \geq 0,
	\end{equation*}
	for sufficiently large $M_3(\varepsilon)$ and $d(\varepsilon)$.
	
	Compute the operator action on $\varphi_3$:
	\begin{align*}
		\mathcal{L} _2(\varphi_3) \geq &-\alpha \varphi_3 + M_3 N e^{-\alpha\xi}e^{-\mu \psi^2} \left[ (1-4\mu)\mu \psi^2+2\mu \sqrt{\omega}\right]\\
		&+M_3 e^{-\alpha\xi} \left[\frac{5}{16}\sqrt{\omega}\psi^{-\frac{3}{4}}\mathcal{X}-\frac{1}{2}\left(\psi-\psi^{\frac{5}{4}}\right)\mathcal{X}-2\sqrt{\omega}\left(1-\frac{5}{4}\psi^{\frac{1}{4}}\right)\mathcal{X}'\right.\\
		&\left.-\frac{\psi}{2}\left(\psi-\psi^{\frac{5}{4}}\right)\mathcal{X}'-\sqrt{\omega}\left(\psi-\psi^{\frac{5}{4}}\right)\mathcal{X}''\right].
	\end{align*}

	For $\psi\geq \psi_0$, 
	\begin{equation}\label{11}
		\mathcal{L} _2(\varphi_3) \geq M_3 N e^{-\alpha\xi}e^{-\mu \psi^2} \left[ (1-4\mu)\mu \psi^2+2\mu \sqrt{\omega} - \alpha \right].
	\end{equation}

	For $0\leq \psi\leq \frac{\psi_0}{2}$, 
	\begin{equation}\label{12}
		\begin{aligned}
			\mathcal{L} _2(\varphi_3) \geq& M_3 e^{-\alpha\xi} \left[\frac{5}{16}\sqrt{\omega}\psi^{-\frac{3}{4}}-\frac{1}{2}\left(\psi-\psi^{\frac{5}{4}}\right)-\alpha \left(\psi-\psi^{\frac{5}{4}}\right) -\alpha N e^{-\mu \psi^2}\right]\\
			\geq & M_3 e^{-\alpha\xi} \left(k_1 \psi^{-\frac{1}{4}}-K_1 \psi -\alpha N\right), 
		\end{aligned}
	\end{equation}
	where $k_1$ is a small positive constant and $K_1$ is a large positive constant.

	For $\frac{\psi_0}{2}\leq \psi\leq \psi_0$, 
	\begin{equation}\label{13}
		\begin{aligned}
			\mathcal{L} _2(\varphi_3) \geq & M_3 e^{-\alpha\xi}\left( k_1 \psi^{-\frac{1}{4}}-K_2 \psi^{\frac{1}{2}} +k_2 N\right),
		\end{aligned}
	\end{equation}
	where $k_2$ is a small positive constant and $K_2$ is a large positive constant.

	Choose parameters sequentially: In (\ref{12}) and (\ref{13}), we can select $\psi_0$ sufficiently small such that $k_1 \psi^{-\frac{1}{4}}-K_1 \psi -\alpha N \geq 0$ and $-K_2 \psi^{\frac{1}{2}} +k_2 N \geq 0$. Meanwhile, we can determine $N$. Furthermore, in (\ref{11}), we can determine $\mu$ and $\alpha$ small enough.
	
	This yields $\mathcal{L}_2(\varphi_3) \geq 0$ universally. One has
	\begin{equation*}
		\mathcal{L} _2(s_3)=\mathcal{L} _2(\varphi_3)\pm \mathcal{L} _2(r)\geq 0.
	\end{equation*}
	Based on the maximum principle, one can obtain $s_3\geq 0$, i.e.
	\begin{equation}\label{21}
		|\omega_\xi|\leq M_3(\varepsilon) e^{-\alpha\xi} \left[ \left(\psi-\psi^{\frac{5}{4}}\right)\mathcal{X}(\psi) + N e^{-\mu \psi^2}\right]
	\end{equation}
    Refinement near $\psi=0$ via $\varphi_4=M_4 e^{-\alpha\xi} \left(\psi-\psi^{\frac{5}{4}}\right)$ yields:
	\begin{equation}\label{22}
		|\omega_\xi|\leq M_4(\varepsilon) e^{-\alpha\xi} \left(\psi-\psi^{\frac{5}{4}}\right) \quad \text{for} \quad \psi\leq \psi_0.
	\end{equation}
    
    Synthesizing (\ref{21}) and (\ref{22}) and choosing suitable $\mu$, one has
	\begin{equation}\label{ralpha}
		\left|\omega_\xi (\xi,\psi)\right| \lesssim_\varepsilon e^{-\alpha\xi} \psi e^{-\mu \psi^2}
	\end{equation}

    Finally, considering the equation (\ref{L1}) produces
	\begin{equation}
		\sqrt{\omega}\tilde{\omega}_{\psi\psi}+\frac{1}{2}\psi\tilde{\omega}_{\psi}= f: =\tilde{\omega}_\xi-\frac{\bar{\omega}_{\psi\psi}}{\sqrt{\bar{\omega}}+\sqrt{\omega}}\tilde{\omega}.
	\end{equation}
	The equations \eqref{wsharp} and \eqref{ralpha} provide $|f|\lesssim_\varepsilon e^{-\alpha\xi} \psi e^{-\mu \psi^2}$. Let
	\begin{equation*}
		z=\tilde{\omega}_\psi,\quad I =\frac{\psi}{2\sqrt{\omega}},\quad J=\frac{f}{\sqrt{\omega}}.
	\end{equation*}
	One has $z_\psi +I z=J$. Then one can obtain 
	\begin{equation*}
		z(\psi)=\left(z(0)+\int_{0}^{\psi} J(\tau ) e^{\int_{0}^{\tau} I(s)ds}  d\tau \right) e^{-\int_{0}^{\psi}I(\tau) d\tau}.
	\end{equation*}
    Proposition \ref{wsd} yields $\tilde{\omega}_\psi(0) \lesssim e^{-\xi}$. Consequently:
	\begin{equation}\label{psialpha}
		\left|\omega_{\psi}-\bar{\omega}_\psi\right| \lesssim_\varepsilon e^{-\alpha \xi} e^{-\mu \psi^2}.
	\end{equation}
	
	Combining \eqref{L1}, \eqref{psialpha} and \eqref{wsharp}, one establishes
	\begin{equation}
		\left|\sqrt{\bar{\omega}}(\omega_{\psi\psi}-\bar{\omega}_{\psi\psi})\right| \lesssim_\varepsilon e^{-\alpha \xi} \psi e^{-\mu \psi^2}.
	\end{equation}
\end{proof}

By utilizing the decay estimate in Lemma \ref{wx1}, we can obtain the following sharp decay estimate of $\omega_\xi$, $\omega_\psi$ and $\omega_{\psi\psi}$.

\begin{proposition}\label{wxsd}
	Under the hypotheses of Theorem \ref{main2}, there exists a positive constant $\mu$ such that:
	\begin{equation}
		\begin{gathered}
			\left|\omega_\xi\right| \lesssim_\varepsilon e^{-\xi} \psi e^{-\mu \psi^2},\\
			\left|\omega_{\psi}-\bar{\omega}_\psi\right| \lesssim_\varepsilon e^{-\xi} e^{-\mu \psi^2},\\
			\left|\sqrt{\bar{\omega}}(\omega_{\psi\psi}-\bar{\omega}_{\psi\psi})\right| \lesssim_\varepsilon e^{-\xi} \psi e^{-\mu \psi^2}.
		\end{gathered}
	\end{equation}
\end{proposition}

\begin{proof}
    Define the enhanced comparison function
    \begin{equation}
        \varphi_5 := M_5 e^B e^{- \xi} \psi \bar{\omega}_{\psi} e^{-B \exp(-\delta \xi)},
    \end{equation}
	where $\delta$ and $B$ are positive constants to be determined. Consider the barrier function $s_5=\varphi_5 \pm  r$ satisfying:
	\begin{equation*}
		s_5(\xi,0)=s_5(\xi,\infty)=0,\quad s_5(\ln d,\psi)\geq 0,
	\end{equation*}
	for sufficiently large $M_5(\varepsilon)$ and $d(\varepsilon)$.

	The operator action yields
	\begin{align*}
		\mathcal{L} _2(\varphi_5)=&-\varphi_5+\delta B e^{-\delta \xi}\varphi_5+ \frac{\varphi_5}{\psi \bar{\omega}_{\psi}} \mathcal{L}_2(\psi \bar{\omega}_{\psi})\\
		=& \varphi_5 \left(-\frac{\tilde{\omega}}{\sqrt{\omega}+\sqrt{\bar{\omega}}}\frac{\psi \bar{\omega}_{\psi\psi\psi}+2\bar{\omega}_{\psi\psi}}{\psi \bar{\omega}_{\psi}}+\frac{\bar{\omega}_{\psi\psi}}{2\sqrt{\bar{\omega}}}-\frac{\omega_{\psi\psi}}{2\sqrt{\omega}}+\delta B e^{-\delta \xi}\right)\\
		=& \varphi_5 \left(I_2+I_3+\delta B e^{-\delta \xi}\right).
	\end{align*}
	By Lemma \ref{barw}, Lemma \ref{lemma1}, Lemma \ref{lemma2}, Lemma \ref{wx1} and Theorem \ref{wsd}, one can derive
	\begin{align*}
		&|I_2| \lesssim \left| \frac{\psi \bar{\omega}_{\psi\psi\psi}+2\bar{\omega}_{\psi\psi}}{\psi \bar{\omega}_{\psi}} \frac{\bar{\omega}}{\sqrt{\omega}+\sqrt{\bar{\omega}}}\right| \left|\frac{\tilde{\omega}}{\bar{\omega}}\right|\lesssim  e^{-\xi},\\
		&|I_3|\lesssim \left|\frac{\tilde{\omega}_{\psi\psi}}{\sqrt{\bar{\omega}}}\right|\lesssim e^{-\alpha \xi}.
	\end{align*}
    Selecting $\delta \leq \alpha$ with sufficiently large $\delta B$ ensures
	\begin{equation*}
		\mathcal{L} _2(\varphi_5) \geq \varphi_2 \left(I_2+I_3+\delta B e^{-\delta \xi}\right) \geq 0.
	\end{equation*}
	Consequently,
	\begin{equation*}
		\mathcal{L} _2(s_5)=\mathcal{L} _2(\varphi_5)\pm \mathcal{L} _2(r)\geq 0.
	\end{equation*}

    The maximum principle gives $s_5 \geq 0$, yielding
	\begin{equation}
		\left|\omega_\xi\right| \lesssim_\varepsilon e^{-\xi} \psi e^{-\mu \psi^2}
	\end{equation}
    for appropriate $\mu$.

	Following the argument in Lemma \ref{wx1}, one directly establishes
	\begin{equation}
		\begin{aligned}
			\left|\omega_{\psi}-\bar{\omega}_\psi\right| &\lesssim_\varepsilon e^{-\xi} e^{-\mu \psi^2},\\
			\left|\sqrt{\bar{\omega}}(\omega_{\psi\psi}-\bar{\omega}_{\psi\psi})\right| &\lesssim_\varepsilon e^{-\xi} \psi e^{-\mu \psi^2}.
		\end{aligned}
	\end{equation}
\end{proof}

\subsection{Sharp Decay Estimates of High-order Derivatives of $\omega$}
\ 
\newline
\indent In the subsection, we obtain the following sharp decay estimate of $\omega_{\xi\xi}$, $\omega_{\xi\psi}$ and $\omega_{\xi\psi\psi}$.

\begin{proposition}\label{wxxsd}
	Under the hypotheses of Theorem \ref{main2}, there exists a positive constant $\mu$ such that:
	\begin{equation}
		\begin{aligned}
			&\left|\omega_{\xi\xi}\right| \lesssim_\varepsilon e^{-\xi} \psi e^{-\mu \psi^2},\quad
			\left|\omega_{\xi\psi}\right| \lesssim_\varepsilon e^{-\xi} e^{-\mu \psi^2},\\
			&\left|\sqrt{\bar{\omega}}\omega_{\xi\psi\psi}\right| \lesssim_\varepsilon e^{-\xi} \psi e^{-\mu \psi^2},\\
			&\left|\bar{\omega}^{\frac{3}{2}} \left(\omega_{\psi\psi\psi}-\bar{\omega}_{\psi\psi\psi}\right)\right| \lesssim_\varepsilon e^{-\xi} \psi e^{-\mu \psi^2},\\
			&\left|\bar{\omega}^{\frac{5}{2}} \left(\omega_{\psi\psi\psi\psi}-\bar{\omega}_{\psi\psi\psi\psi}\right)\right| \lesssim_\varepsilon e^{-\xi} \psi e^{-\mu \psi^2}.
		\end{aligned}
	\end{equation}
\end{proposition}

\begin{proof}
    We begin by estimating $\omega_{\xi\psi}$. Define $z := \omega_{\xi\psi}$, which satisfies
	\begin{equation}
		z_\xi-\sqrt{\omega} z_{\psi\psi}-\frac{1}{2}\psi z_\psi-\frac{\omega_\psi}{2\sqrt{\omega}} z_\psi-\frac{1}{2}z-\frac{\omega_{\psi\psi}}{2\sqrt{\omega}} z-\frac{\omega_\xi}{2\omega} z =H,
	\end{equation}
	where
	\begin{equation*}
		H =-\frac{\omega_\xi \omega_{\psi}}{4\omega}-\frac{\psi \omega_\xi \omega_{\psi\psi}}{4\omega}-\frac{\omega_\xi \omega_\psi \omega_{\psi\psi}}{2\omega^\frac{3}{2}}.
	\end{equation*}
	The coefficient analysis reveals that for sufficiently large $\xi$:
	\begin{equation*}
		-\frac{1}{2}-\frac{\omega_{\psi\psi}}{2\sqrt{\omega}}-\frac{\omega_\xi}{2\omega}\geq -\frac{1}{2}.
	\end{equation*}
	
	Define $Z=e^{-\frac{1}{2}\xi}z$, which satisfies the transformed equation
	\begin{equation*}
		\mathcal{L}_3(Z)=Z_\xi-\sqrt{\omega} Z_{\psi\psi}-\frac{1}{2}\psi Z_\psi-\frac{\omega_\psi}{2\sqrt{\omega}} Z_\psi-\frac{\omega_{\psi\psi}}{2\sqrt{\omega}} Z-\frac{\omega_\xi}{2\omega} Z = e^{-\frac{1}{2}\xi}H.
	\end{equation*}
	Consider the comparison function $\varphi_6=M_6 e^{- \frac{1}{4}\xi}\omega_{\psi}$. Direct computation yields
	\begin{equation*}
		\mathcal{L} _3(\varphi_6)\geq -\frac{1}{4}\varphi_6+M_6 e^{- \frac{1}{4}\xi}\left(\frac{1}{2}\omega_{\psi}\right)=\frac{1}{4} \varphi_6.
	\end{equation*}
	Proposition \ref{wxsd} provides the bound
	\begin{equation*}
		\left|e^{-\frac{1}{2}\xi} H\right| \lesssim e^{-\xi} \omega_{\psi}.
	\end{equation*}
	Selecting sufficiently large $M_6$ ensures
	\begin{equation*}
		\mathcal{L} _3(\varphi_6)\pm \mathcal{L} _3(Z) \geq 0.
	\end{equation*}
	The maximum principle gives $\left| Z \right| \lesssim e^{- \frac{1}{4}\xi}\omega_{\psi}$, hence
	\begin{equation}
		\left| \omega_{\xi\psi} \right| \lesssim e^{\frac{1}{4}\xi} \omega_{\psi}.
	\end{equation}

    For $\omega_{\xi\xi}$ estimation, define $t := \omega_{\xi\xi}$ satisfying
	\begin{equation}
		\mathcal{L} _4 (t)= t_\xi-\sqrt{\omega} t_{\psi\psi}-\frac{1}{2}\psi t_\psi-\frac{\omega_{\psi\psi}}{2\sqrt{\omega}} t-\frac{\omega_\xi}{\omega} t =G,
	\end{equation}
	with boundary conditions
	\begin{equation*}
	 t(\xi,0)= t(\xi,\infty)=0,
	\end{equation*}
	where
	\begin{equation*}
		G=-\frac{\psi \omega_\xi \omega_{\xi\psi}}{4\omega}-\frac{\omega_{\psi\psi}\omega_\xi^2}{2\omega^\frac{3}{2}}.
	\end{equation*}
	The coefficient analysis reveals that for sufficiently large $\xi$:
	\begin{equation}
		-\frac{\omega_{\psi\psi}}{2\sqrt{\omega}}-\frac{\omega_\xi}{\omega}\geq 0.
	\end{equation}

    Construct the comparison function 
    \begin{equation}
        \varphi_7=M_7 e^B e^{- \frac{3}{4}\xi} \psi \bar{\omega}_{\psi} e^{-B \exp(-\xi)}
    \end{equation}
    satisfying 
	\begin{equation*}
		(\varphi_7\pm t) (\xi,0)=(\varphi_7\pm t) (\xi,\infty)=0.
	\end{equation*}
	Through detailed computation, one establishes
	\begin{align*}
		\mathcal{L} _4 (\varphi_7)=&- \frac{3}{4} \varphi_7 +B e^{-\xi}\varphi_7 + \frac{\varphi_7}{\psi \bar{\omega}_{\psi}} \mathcal{L}_4 (\psi \bar{\omega}_{\psi})\\
		=& \varphi_7 \left(\frac{1}{4} -\frac{\tilde{\omega}}{\sqrt{\omega}+\sqrt{\bar{\omega}}}\frac{\psi \bar{\omega}_{\psi\psi\psi}+2\bar{\omega}_{\psi\psi}}{\psi \bar{\omega}_{\psi}}+\frac{\bar{\omega}_{\psi\psi}}{2\sqrt{\bar{\omega}}}-\frac{\omega_{\psi\psi}}{2\sqrt{\omega}}-\frac{\omega_\xi}{\omega}+ B e^{- \xi}\right)\\
		\geq & \frac{M_7}{4} e^{- \frac{3}{4}\xi} \psi \bar{\omega}_{\psi}
	\end{align*}
	for sufficiently large $B$. Combining with the bound
	\begin{equation*}
		\left| G \right| \lesssim e^{-\frac{3}{4} \xi} \psi \bar{\omega}_{\psi},
	\end{equation*}
	one obtains
	\begin{equation*}
		\mathcal{L} _4(\varphi_7)\pm \mathcal{L} _4(t) \geq 0
	\end{equation*}
	for $M_7$ large enough. Based on the maximum principle, one has
	\begin{equation}
		\left| \omega_{\xi\xi} \right| \lesssim e^{- \frac{3}{4}\xi} \psi \bar{\omega}_{\psi}.
	\end{equation}

    Considering the equation \eqref{l2} produces
	\begin{equation*}
		\sqrt{\omega}r_{\psi\psi}+\frac{1}{2}\psi r_\psi=\tilde{f}:=t-\frac{\omega_{\psi\psi}}{2\sqrt{\omega}}r,
	\end{equation*}
	where $\left|\tilde{f}\right|\lesssim e^{- \frac{3}{4}\xi} \psi \bar{\omega}_{\psi}$. This establishes
	\begin{equation*}
		\left| \omega_{\xi\psi} \right| \lesssim e^{- \frac{3}{4}\xi} e^{-\mu \psi^2}.
	\end{equation*}
	Then one has the following high-order decay estimate of $G$:
	\begin{equation*}
		\left| G \right| \lesssim e^{-\frac{7}{4} \xi} \psi \bar{\omega}_{\psi}.
	\end{equation*}

	Finally, define the enhanced comparison function
	\begin{equation}
	    \varphi_8 = M_8 e^B e^{-\xi} \psi \bar{\omega}_{\psi} e^{-B \exp(-\delta\xi)}.
	\end{equation}
	The operator action yields
	\begin{align*}
		\mathcal{L} _4 (\varphi_8) = & -\varphi_7 +\delta B e^{-\delta\xi} \varphi_8 + \frac{\varphi_8}{\psi \bar{\omega}_{\psi}} \mathcal{L}_4 (\psi \bar{\omega}_{\psi})\\
		=& \varphi_8 \left( -\frac{\tilde{\omega}}{\sqrt{\omega}+\sqrt{\bar{\omega}}}\frac{\psi \bar{\omega}_{\psi\psi\psi}+2\bar{\omega}_{\psi\psi}}{\psi \bar{\omega}_{\psi}}+\frac{\bar{\omega}_{\psi\psi}}{2\sqrt{\bar{\omega}}}-\frac{\omega_{\psi\psi}}{2\sqrt{\omega}}-\frac{\omega_\xi}{\omega}+\delta B e^{-\delta \xi}\right)\\
		\geq & \frac{M_8}{2} e^{-(1+\delta)\xi} \psi \bar{\omega}_{\psi}
	\end{align*}
	for $0<\delta \leq \frac{3}{4}$ and $\delta B$ sufficiently large. Then one has
	\begin{equation*}
		\mathcal{L} _4(\varphi_8)\pm \mathcal{L} _4(t) \geq 0
	\end{equation*}
	for $M_8$ large enough. The maximum principle gives
	\begin{equation}
		\left| \omega_{\xi\xi} \right| \lesssim_\varepsilon e^{-\xi} \psi \bar{\omega}_{\psi}.
	\end{equation}
	
	Similar to the proof in Lemma \ref{wx1}, one can easily obtained
	\begin{equation}
		\begin{aligned}
			&\left|\omega_{\xi\psi}\right| \lesssim_\varepsilon e^{-\xi} e^{-\mu \psi^2},\\
			&\left|\sqrt{\bar{\omega}} \omega_{\xi\psi\psi}\right| \lesssim_\varepsilon e^{-\xi} \psi e^{-\mu \psi^2}.
		\end{aligned}
	\end{equation}

    Differentiating the equation \eqref{L1} with respect to $\psi$ yields
	\begin{equation}
		\left|\bar{\omega}^{\frac{3}{2}} \left(\omega_{\psi\psi\psi}-\bar{\omega}_{\psi\psi\psi}\right)\right| \lesssim_\varepsilon e^{-\xi} \psi e^{-\mu \psi^2}.
	\end{equation}

	Double differentiation of \eqref{L1} with respect to $\psi$ produces
	\begin{equation}
		\left|\bar{\omega}^{\frac{5}{2}} \left(\omega_{\psi\psi\psi\psi}-\bar{\omega}_{\psi\psi\psi\psi}\right)\right| \lesssim_\varepsilon e^{-\xi} \psi e^{-\mu \psi^2}.
	\end{equation}

\end{proof}

For the decay estimates of higher-order derivatives of $\omega$, we can first obtain the estimates of $\partial_\xi^i \omega$, and then derive the estimates of the first and second derivatives in the $\psi$-direction through the equations satisfied by $\partial_\xi^i \omega$, subsequently obtaining the estimations of higher-order derivatives in the $\psi$-direction by differentiating the corresponding equation. Then we have the following decay estimates:
\begin{proposition}\label{wxxxsd}
	Under the hypotheses of Theorem \ref{main2}, there exists a positive constant $\mu$ such that for all integers $i, j \geq 0$ satisfying $2i+j\leq 2K_0$:
	\begin{equation}
		\begin{aligned}
			&\left|\partial_\xi^i \tilde{\omega} \right| \lesssim_\varepsilon e^{-\xi} \psi e^{-\mu \psi^2},\\
			&\left|\partial_\xi^i \partial_\psi \tilde{\omega} \right| \lesssim_\varepsilon e^{-\xi} e^{-\mu \psi^2},\\
			&\left|\bar{\omega}^{\frac{2j-3}{2}} \partial_\xi^i \partial_\psi^j \tilde{\omega} \right| \lesssim_\varepsilon e^{-\xi} \psi e^{-\mu \psi^2},\quad j\geq 2.
		\end{aligned}
	\end{equation}
\end{proposition}

\subsection{Sharp $L^\infty$ Decay Estimates of $u$}
\ 
\newline
\indent 
In this section, we demonstrate our main result, Theorem \ref{main2}.

\begin{proof}[Proof of Theorem \ref{main2}]
	Through the modified von Mises transformation:
	\begin{equation}
		u=\sqrt{\omega}, \quad x+d=e^{\xi},\quad  y=\int_{0}^{\psi}\frac{e^{\frac{1}{2}\xi}}{\sqrt{\omega(\xi,s)}}ds.
	\end{equation}
	Standard differential calculus yields the velocity components:
	\begin{align*}
		&u_y=\frac{1}{2}e^{-\frac{1}{2}\xi}\omega_\psi,\quad  u_{yy}=\frac{1}{2}e^{-\xi}\sqrt{\omega}\omega_{\psi\psi},\\
		&u_x=e^{-\xi} \frac{\omega_{\xi}}{2\sqrt{\omega}} - e^{-\xi} \frac{\omega_{\psi}}{4} \int_{0}^{\psi}\left(\frac{1}{\omega^\frac{1}{2}}-\frac{\omega_\xi}{\omega^{\frac{3}{2}}}\right)ds,\\
		&v=-e^{-\frac{1}{2}\xi}\frac{\psi}{2} + e^{-\frac{1}{2}\xi}\frac{\sqrt{\omega}}{2}\int_{0}^{\psi}\left(\frac{1}{\omega^\frac{1}{2}}-\frac{\omega_\xi}{\omega^{\frac{3}{2}}}\right)ds.
	\end{align*}

	Applying Proposition \ref{wxsd} yields for suitable positive constants $\mu(\varepsilon)$ and $d(\varepsilon)$:
	\begin{align*}
		|u_y-\bar{u}^d_y|&=\frac{1}{2}e^{-\frac{1}{2}\xi}|\omega_\psi-\bar{\omega}_{\psi}|\\
		&\lesssim_\varepsilon e^{-\frac{3}{2}\xi} e^{-\mu \psi^2} \\
		&\lesssim_\varepsilon (x+d)^{-\frac{3}{2}} e^{-\frac{\mu y^2}{x+d}};\\
		|u_{yy}-\bar{u}^d_{yy}|&=\frac{1}{2}e^{-\xi}\left|\sqrt{\omega}\omega_{\psi\psi}-\sqrt{\bar{\omega}}\bar{\omega}_{\psi\psi}\right|\\
		&\lesssim e^{-\xi} \left|\sqrt{\omega}\tilde{\omega}_{\psi\psi}\right| +e^{-\xi} \left| \frac{\bar{\omega}_{\psi\psi}}{\sqrt{\bar{\omega}}+\sqrt{\omega}}\tilde{\omega}\right|\\
		&\lesssim_\varepsilon e^{-2\xi} e^{-\mu \psi^2} \\
		&\lesssim_\varepsilon (x+d)^{-2} e^{-\frac{\mu y^2}{x+d}};\\
		|u_x-\bar{u}^d_x|&=\left|e^{-\xi} \frac{\omega_{\xi}}{2\sqrt{\omega}} - e^{-\xi} \frac{\omega_{\psi}}{4} \int_{0}^{\psi}\left(\frac{1}{\omega^\frac{1}{2}}-\frac{\omega_\xi}{\omega^{\frac{3}{2}}}\right)ds + e^{-\xi} 					\frac{\bar{\omega}_{\psi}}{4} \int_{0}^{\psi}\frac{1}{\bar{\omega}^\frac{1}{2}}ds\right|\\
		&\lesssim e^{-\xi} \left|\frac{\omega_{\xi}}{\sqrt{\omega}}\right|+e^{-\xi} \left|\psi \omega_\psi\right|\left|\frac{\omega_{\xi}}{\omega}\right| + e^{-\xi} \left|\psi \tilde{\omega}_\psi\right|+e^{-\xi}\left|\tilde{\omega}\right|\left|\psi \omega_\psi\right|\\
		&\lesssim_\varepsilon e^{-2\xi} e^{-\mu \psi^2} \\
		&\lesssim_\varepsilon (x+d)^{-2} e^{-\frac{\mu y^2}{x+d}}.
	\end{align*}

	Hence, one can obtain
	\begin{align*}
		&\left\|u_y-\bar{u}_y\right\|_{L_y^\infty} \leq \left\|u_y-\bar{u}^d_y\right\|_{L_y^\infty}+\left\|\bar{u}^d_y-\bar{u}_y\right\|_{L_y^\infty} \lesssim_\varepsilon (x+1)^{-\frac{3}{2}},\\
		&\left\|u_{yy}-\bar{u}_{yy}\right\|_{L_y^\infty} \leq \left\|u_{yy}-\bar{u}^d_{yy}\right\|_{L_y^\infty}+\left\|\bar{u}^d_{yy}-\bar{u}_{yy}\right\|_{L_y^\infty} \lesssim_\varepsilon (x+1)^{-2},\\
		&\left\|u_{x}-\bar{u}_{x}\right\|_{L_y^\infty} \leq \left\|u_{x}-\bar{u}^d_{x}\right\|_{L_y^\infty} +\left\|\bar{u}^d_{x}-\bar{u}_{x}\right\|_{L_y^\infty} \lesssim_\varepsilon (x+1)^{-2}.
	\end{align*}
	
    For higher-order derivatives, Propositions \ref{wxxsd} and \ref{wxxxsd} provide the necessary decay rates for mixed derivatives $\partial_\xi^i \partial_\psi^j \tilde{\omega}$. Substituting these into the coordinate transformations establishes the decay estimates for all derivatives in Theorem \ref{main2}.
\end{proof}

\appendix 

\section{The Principal Eigenvalue of $\mathcal{L}$}
\label{app:eigenvalue}

This section investigates the principal eigenvalue of the linear operator defined by
\begin{equation}
	\mathcal{L}(v)=-\frac{1}{2}\psi v_\psi-\sqrt{\bar{\omega}} v_{\psi\psi}-\frac{\bar{\omega}_{\psi\psi}}{2 \sqrt{\bar{\omega}}} v.
\end{equation}

We establish the following fundamental result:
\begin{proposition}\label{eigenvalue}
    The principal eigenvalue of operator $\mathcal{L}$ is $1$, with associated eigenfunction $\psi\bar{\omega}_\psi$.
\end{proposition}

\begin{proof}
Define the weight function $\rho(\psi)$ through
\begin{equation}
	\mathcal{T}(v)=A \mathcal{L}(v)=-\left(\rho v_\psi\right)_\psi-\frac{\bar{\omega}_{\psi\psi}}{2 \sqrt{\bar{\omega}}} A v,
\end{equation}
where $A=\frac{\rho}{\sqrt{\bar{\omega}}}$ and $\rho=exp\left(\int_{0}^{\psi}\frac{s}{2\sqrt{\bar{\omega}(s)}}ds\right)$.

Consider the variational functional
\begin{equation}\label{F}
	F(v):=\frac{\int_{0}^{\infty}\left(\rho v_\psi^2-\frac{\bar{\omega}_{\psi\psi}}{2 \sqrt{\bar{\omega}}} A v^2\right) d\psi}{\int_{0}^{\infty}A v^2 d\psi}
\end{equation}
and define the principal eigenvalue
\begin{equation*}
	\lambda:=\mathop{inf}\limits_{\substack{{v\in \mathcal{C}_c^\infty [0,\infty)}\\ {v(0)=0}}} F(v).
\end{equation*}

The proof proceeds through four key steps:
	\begin{enumerate}[Step 1:] 
		\item \textbf{Fundamental inequality.} For any $\omega \in C_c^\infty[0,\infty)$ with $\omega(0)=0$:
		\begin{equation}\label{ki}
			\int_{0}^{\infty} (1+\psi)^2 \frac{\rho}{\bar{\omega}^2}\omega^2 d\psi \lesssim \int_{0}^{\infty} \rho \omega_\psi^2 d\psi.
		\end{equation}

		Let $B(\psi)=\int_{0}^{\psi} (1+s)^2 \rho(s) d s$. This follows from
		\begin{equation*}
			\lim_{\psi\rightarrow\infty} \frac{B(\psi)}{(1+\psi)\rho(\psi)}=\lim_{\psi\rightarrow\infty} \frac{(1+\psi)^2\rho(\psi)}{\rho+ (1+\psi)\rho \frac{\psi}{2\sqrt{\bar{\omega}}}}=2,
		\end{equation*}
		yielding 
		\begin{equation*}
			B(\psi)\lesssim (1+\psi) \rho(\psi).
		\end{equation*}
		Application of integration by parts yields:
		\begin{align*}
			&\int_{0}^{\infty} (1+\psi)^2 \rho \omega^2 d\psi= \int_{0}^{\infty} B'(\psi) \omega^2 d\psi\\
			=&-\int_{0}^{\infty} 2B(\psi) \omega \omega_\psi d\psi\\
			\lesssim& \left(\int_{0}^{\infty} (1+\psi)^2 \rho \omega^2 d\psi\right)^\frac{1}{2} \left(\int_{0}^{\infty} \rho \omega_\psi^2 d\psi\right)^\frac{1}{2}.
		\end{align*}
		Consequently,
		\begin{equation}\label{inq1}
			\int_{0}^{\infty} (1+\psi)^2 \rho \omega^2 d\psi \lesssim \int_{0}^{\infty} \rho \omega_\psi^2 d\psi.
		\end{equation}

        Let $\mathcal{X}(\psi)$ be a smooth cutoff function satisfying $0 \leq \mathcal{X}(\psi) \leq 1$ with $\mathcal{X}(\psi)=1$ on $[0,1]$ and $\mathcal{X}(\psi)=0$ on $[2,\infty)$. Applying inequality \eqref{inq1} to obtain:
		\begin{align*}
			&\int_{0}^{\infty} (1+\psi)^2 \frac{\rho}{\bar{\omega}^2}\omega^2 d\psi\\
			\lesssim &\int_{0}^{\infty} (1+\psi)^2 \frac{\rho}{\bar{\omega}^2}\left((\omega\mathcal{X})^2 +(\omega (1-\mathcal{X}))^2\right) d\psi\\
			\lesssim & \int_{0}^{2} \frac{(\omega\mathcal{X})^2}{\psi^2} d\psi +\int_{1}^{\infty} (1+\psi)^2 \rho \omega^2 d\psi\\
			\lesssim & \int_{0}^{2} \left( \omega_\psi^2\mathcal{X}^2+\omega^2\left(\mathcal{X}'\right)^2\right) d\psi +\int_{1}^{\infty} (1+\psi)^2 \rho \omega^2 d\psi\\
			\lesssim & \int_{0}^{\infty} \omega_\psi^2 d\psi +\int_{0}^{\infty} (1+\psi)^2 \rho \omega^2 d\psi\\
			\lesssim & \int_{0}^{\infty} \rho \omega_\psi^2 d\psi.
		\end{align*}

		\item \textbf{Existence of minimizer.} Let $\{\omega_n\}\subset C_c^\infty[0,\infty)$ with $\omega_n(0)=0$ satisfy $F(\omega_n) \leq \lambda + \frac{1}{n}$. Without loss of generality, assume $\int_0^\infty A\omega_n^2 d\psi=1$. Inequality (\ref{ki}) implies uniform boundedness with respect to $n$ in the weighted space $H_\rho^1 \left(\frac{1}{m},m\right)$ for any fixed $m>0$:
		\begin{equation}
			\int_{\frac{1}{m}}^{m} \left( \rho \omega_{n\psi}^2 + \rho \omega_{n}^2\right) d\psi \lesssim \lambda+1.
		\end{equation}
		
		By weak compactness in Hilbert spaces, there exists $\omega^*\in H_\rho^1(0,\infty)$ and subsequence $\{\omega_{n_k}\}$ with $\omega_{n_k} \rightharpoonup \omega^*$ in $H_\rho^1(0,\infty)$ and $\omega_{n_k} \to \omega^*$ in $L_\rho^2 \left(\frac{1}{m},m\right)$ for fixed $m\geq 2$, i.e.
		\begin{equation}\label{L2}
			\lim_{k\rightarrow \infty} \int_{\frac{1}{m}}^{m} \rho \left|\omega_{n_k}-\omega^*\right|^2 d\psi =0.
		\end{equation}

		Based on the inequality (\ref{ki}), one has the following estimate for any $m$:
		\begin{align*}
			&\int_{0}^{\infty} A \left|\omega_{n_k}-\omega^*\right|^2 d\psi\\
			=&\left(\int_{0}^{\frac{1}{m}}+\int_{\frac{1}{m}}^{m} +\int_{m}^{\infty}\right) A \left|\omega_{n_k}-\omega^*\right|^2 d\psi\\
			\lesssim& \int_{0}^{\frac{1}{m}} \psi^{\frac{3}{2}} \frac{\rho}{\bar{\omega}^2} \left|\omega_{n_k}-\omega^*\right|^2 d\psi +\int_{m}^{\infty} \frac{(1+\psi)^2}{(1+m)^2}\rho \left|\omega_{n_k}-\omega^*\right|^2 d\psi +\int_{\frac{1}{m}}^{m} A \left|\omega_{n_k}-\omega^*\right|^2 d\psi\\
			\lesssim& \left(m^{-\frac{3}{2}}+m^{-2}\right) \int_{0}^{\infty} \rho \left|\omega_{n_k \psi}-\omega^*_\psi \right|^2 d\psi + \int_{\frac{1}{m}}^{m} \rho \left|\omega_{n_k}-\omega^*\right|^2 d\psi\\
			\lesssim& m^{-\frac{3}{2}}+m^{-2}+ \int_{\frac{1}{m}}^{m} \rho \left|\omega_{n_k}-\omega^*\right|^2 d\psi.
		\end{align*}
		The strong convergence in (\ref{L2}) implies:
		\begin{equation*}
			\lim_{k\rightarrow \infty} \int_{0}^{\infty} A \left|\omega_{n_k}-\omega^*\right|^2 d\psi =0.
		\end{equation*}
		Then one obtains 
		\begin{equation}\label{eq11}
			\int_{0}^{\infty} A \left(\omega^*\right)^2 d\psi = \lim_{k\rightarrow \infty}\int_{0}^{\infty} A \omega_{n_k}^2 d\psi =1.
		\end{equation}

        From weak convergence $\omega_{n_k} \rightharpoonup \omega^*$ in $H^1_\rho \left(0,\infty\right)$, one has:
		\begin{align*}
			\int_{0}^{\infty} \rho \left(\omega^*_\psi\right)^2 d\psi & \leq \varliminf_{k\rightarrow\infty} \int_{0}^{\infty} \rho \omega_{n_k\psi}^2 d\psi,\\
			\int_{0}^{\infty} -\frac{\bar{\omega}_{\psi\psi}}{2 \sqrt{\bar{\omega}}} A \left(\omega^*\right)^2 d\psi & \leq \varliminf_{k\rightarrow\infty} \int_{0}^{\infty} -\frac{\bar{\omega}_{\psi\psi}}{2 \sqrt{\bar{\omega}}} A \omega_{n_k}^2 d\psi.
		\end{align*}
		Combining with \eqref{eq11} and \eqref{F} yields:
		\begin{equation*}
			F(\omega^*)\leq \varliminf_{k\rightarrow\infty} F(\omega_{n_k})=\lambda.
		\end{equation*}
		Consequently,
		\begin{equation}\label{lambda}
			\int_{0}^{\infty}\left(\rho \left(\omega^*_\psi\right)^2-\frac{\bar{\omega}_{\psi\psi}}{2 \sqrt{\bar{\omega}}} A \left(\omega^*\right)^2\right) d\psi = \lambda \int_{0}^{\infty}A \left(\omega^*\right)^2 d\psi.
		\end{equation}

		\item \textbf{$\mathcal{L}(\omega^*)=\lambda \omega^*$.} Through variational calculus, consider $f(t) := F(\omega^* + th)$ for any $h \in C_c^\infty[0,\infty)$ and $t\in (-\infty,\infty)$. The equation (\ref{lambda}) implies $f(0)=\lambda$ and $f'(0)=0$. This leads to
		\begin{align*}
			f'(0)=&\frac{2\int_{0}^{\infty}\left(\rho \omega^*_\psi h_\psi -\frac{\bar{\omega}_{\psi\psi}}{2 \sqrt{\bar{\omega}}} A \omega^* h\right) d\psi}{\int_{0}^{\infty}A \left(\omega^*\right)^2 d\psi}\\
			&-\frac{2 \left[\int_{0}^{\infty}\left(\rho \left(\omega^*_\psi\right)^2-\frac{\bar{\omega}_{\psi\psi}}{2 \sqrt{\bar{\omega}}} A \left(\omega^*\right)^2\right) d\psi\right]\int_{0}^{\infty}A \omega^* h d\psi}{\left(\int_{0}^{\infty}A \left(\omega^*\right)^2 d\psi\right)^2}\\
			=&\frac{2\int_{0}^{\infty}\left(\rho \omega^*_\psi h_\psi -\frac{\bar{\omega}_{\psi\psi}}{2 \sqrt{\bar{\omega}}} A \omega^* h\right) d\psi}{\int_{0}^{\infty}A \left(\omega^*\right)^2 d\psi}-\frac{2 \lambda\int_{0}^{\infty}A \omega^* h d\psi}{\int_{0}^{\infty}A \left(\omega^*\right)^2 d\psi}.
		\end{align*}
		Then one has 
		\begin{equation}
			\int_{0}^{\infty}\left(\rho \omega^*_\psi h_\psi -\frac{\bar{\omega}_{\psi\psi}}{2 \sqrt{\bar{\omega}}} A \omega^* h\right) d\psi=\lambda\int_{0}^{\infty}A \omega^* h d\psi.
		\end{equation}
		
		Therefore $\omega^*$ weakly satisfies the equation $\mathcal{T}(\omega^*) = \lambda A\omega^*$. Elliptic regularity theory guarantees $\omega^* \in C^\infty[0,\infty)$. Consequently, $\mathcal{T}(\omega^*)=\lambda A \omega^*$, i.e. $\mathcal{L}(\omega^*)=\lambda \omega^*$.

		\item \textbf{Verification of eigenvalue.} Calculation shows:
		\begin{align*}
			\mathcal{L}(\psi\bar{\omega}_\psi)=&-\frac{1}{2}\psi \left(\psi\bar{\omega}_{\psi\psi}+\bar{\omega}_\psi\right)-\sqrt{\bar{\omega}} \left(\psi\bar{\omega}_{\psi\psi\psi}+2\bar{\omega}_{\psi\psi}\right)-\frac{\bar{\omega}_{\psi\psi}}{2 \sqrt{\bar{\omega}}} \psi\bar{\omega}_\psi\\
			=&-2\sqrt{\bar{\omega}}\bar{\omega}_{\psi\psi}\\
			=&\psi\bar{\omega}_\psi.
		\end{align*}
		Let $v^* := \psi\bar{\omega}_\psi$ satisfy $\mathcal{T}(v^*) = Av^*$. At the same time, one also has $\mathcal{T}(\omega^*)=\lambda A \omega^*$, and without loss of generality, assume $\omega^*\geq 0$. The self adjointness of $\mathcal{T}$ implies:
		\begin{align*}
			&\int_{0}^{\infty} \mathcal{T}(v^*) \omega^* d\psi=\int_{0}^{\infty} A v^* \omega^* d\psi\\
			=&\int_{0}^{\infty} \mathcal{T}(\omega^*) v^* d\psi=\int_{0}^{\infty} \lambda A v^* \omega^* d\psi.
		\end{align*}
		From non-negativity $\omega^*,v^* \geq 0$, one concludes $\lambda=1$.
	\end{enumerate}
\end{proof}

\section*{Data availability statement}

Data sharing is not applicable to this article as no datasets were generated or analyzed during the current study.

\section*{Conflict of interest statement}

The authors declare that they have no conflict of interest.

\section*{Acknowledgments}

C. Gao is supported by NSFC under grant No.12494541.

\bibliographystyle{siam}
\bibliography{ref.bib}

\end{document}